\documentclass{amsart}
\usepackage{amsmath,amssymb,amsthm,amsfonts,epsfig,verbatim}
\input{xy}                                                                      
\xyoption{all}
\CompileMatrices 

\newcommand{\com}[1]{
}

\newcommand{\bp}{\begin{pmatrix}}
\newcommand{\ep}{\end{pmatrix}}
\newcommand{\be}{\begin{equation}}
\newcommand{\ee}{\end{equation}}
\newcommand{\bs}{\begin{split}}
\newcommand{\es}{\end{split}}
\newcommand{\bc}{\begin{center}}
\newcommand{\ec}{\end{center}}
\newcommand{\ed}{{\rm d}}
\newcommand{\w}{{\mathchoice{\,{\scriptstyle\wedge}\,}{{\scriptstyle\wedge}}    
      {{\scriptscriptstyle\wedge}}{{\scriptscriptstyle\wedge}}}}                
\newcommand{\lhk}{\mathbin{\hbox{\vrule height1.4pt width4pt depth-1pt          
             \vrule height4pt width0.4pt depth-1pt}}}  
\newcommand{\ol}{\overline}
\newcommand{\liealgebra}[1]{{\mathfrak {#1}}}
\newcommand{\g}{\liealgebra{g}}
\newcommand{\ka}{\liealgebra{k}}
\newcommand{\gl}{\liealgebra{gl}}
\newcommand{\sla}{\liealgebra{sl}}

\newcommand{\su}{\liealgebra{su}}

\newcommand{\liegroup}[1]{{\operatorname{#1}}}
\newcommand{\G}{\liegroup{G}}
\newcommand{\K}{\liegroup{K}}

\newcommand{\SL}{\liegroup{SL}}
\newcommand{\SO}{\liegroup{SO}}
\newcommand{\SU}{\liegroup{SU}}

\renewcommand{\Im}{\operatorname{Im}}
\newcommand{\mo}{\sqrt{-1}}

\providecommand{\Imag}{\mathop{\rm Im}\nolimits}
\newcommand{\R}{\mathbb R}
\newcommand{\C}{\mathbb C}

\newcommand{\Z}{\mathbb Z}
\newcommand{\N}{\mathbb N}

\newcommand{\s}[1]{{\mathbb S}^{#1}}

\newcommand{\mcc}{\mathcal C}
\newcommand{\mcd}{\mathcal D}
\newcommand{\mce}{\mathcal E}

\newcommand{\mci}{\mathcal I}

\newcommand{\mcl}{\mathcal L}

\newcommand{\mcn}{\mathcal N}

\newcommand{\mcp}{\mathcal P}

\newcommand{\mct}{\mathcal T}

\newcommand{\p}{\varphi}
\newcommand{\pt}{\tilde{\varphi}}
\newcommand{\ph}{\hat \varphi}

\newcommand{\del}{{\partial}}
\newcommand{\delb}{\bar{\partial}}

\newcommand{\ns}{\negthinspace}
\newcommand{\trp}{\; {}^t \ns}

\newcommand{\M}[1]{M^{(#1)}}
\providecommand{\mcip}[1]{{\mci}^{(#1)}}

\providecommand{\wt}{{\rm {wd}}}

\newcommand{\dd}[2]{\frac{\partial {#1}}{\partial {#2}}}

\newcommand{\om}{\omega}

\newcommand{\Om}{\Omega}

\newcommand{\I}[1]{{\rm I}^{(#1)}}
\newcommand{\zb}{\ol{z}}

\newcommand{\ub}{\ol{u}}

\newcommand{\xib}{\ol{\xi}}

\newcommand{\zetab}{\ol{\zeta}}

\newcommand{\bh}{\hat{b}}

\usepackage[mathscr]{eucal}
\usepackage{bbm}
\usepackage{oldgerm}

\theoremstyle{plain}
\newtheorem{theorem}{Theorem}[section]
\newtheorem{lemma}[theorem]{Lemma}

\newtheorem{proposition}[theorem]{Proposition}

\newtheorem{rem}[theorem]{Remark}
\newtheorem{definition}{Definition}[section]

\newtheorem{example}[theorem]{Example}

\begin{document}
\title[Killing fields and conservation laws ]{Killing fields and conservation laws for rank-one Toda field equations}
\author{Daniel Fox}
\address{Daniel Fox, W2-7 Department of Mathematics, 1700 Spring Garden Street, Community College of Philadelphia, Philadelphia, PA 19130}
\email{foxdanie@gmail.com}
\subjclass[2000]{Primary 53,35}
\date{\today}
\thanks{I would like to thank  Fran Burstall, Chuu-Lian Terng, Dominic Joyce, and Oliver Goertsches for useful discussions on this topic and Emma Carberry and Joe S. Wang for important comments on an earlier draft.  }

\begin{abstract}
We present a connection between the Killing fields that arise in the loop-group approach to integrable systems and conservation laws viewed as elements of the characteristic cohomology.  We use the connection to generate the complete set of conservation laws (as elements of the characteristic cohomology) for the Tzitzeica equation, completing the work in \cite{Fox2011}.  

We define a notion of \emph{finite-type} for integral manifolds of exterior differential systems directly in terms of conservation laws that generalizes the definition of Pinkall-Sterling \cite{Pinkall1989}.  The definition applies to any exterior differential system that has infinitely many conservation laws possessing a normal form.   

Finally, we show that, for the rank-one Toda field equations, every characteristic cohomology class has a translation invariant representative as an undifferentiated conservation law.  Therefore the characteristic cohomology defines de Rham cohomology classes on doubly periodic solutions.  
\end{abstract}
\maketitle

\tableofcontents



%
%
%

\section{Introduction}
There are various approaches to conservation laws of differential equations (for example, see \cite{Bryant1995,Shadwick1981,Terng2005,Sanders2001}).  This article attempts to further the theory of characteristic cohomology developed by Bryant and Griffiths \cite{Bryant1995}, continuing the work in \cite{Fox2011}.  The characteristic cohomology approach emphasizes 1)the cohomological nature of conservation laws, 2)a universal perspective, and 3)geometric representatives similar to harmonic representatives of de Rham classes.   

In this article we establish a connection between Killing fields \cite{Burstall1993,Burstall1995} and the characteristic cohomology \cite{Bryant1995} of Toda field equations \cite{Bolton1995} and use it to completely determine the characteristic cohomology of the Tzitzeica equation (which corresponds to the Toda field equation for $\SU(3)/\SO(2)$).  We expect the method applied to $\SU(3)/\SO(2)$ to generalize to all Toda field equations and even to all primitive map systems, though we have not attempted this. 
   
Terng and Wang \cite{Terng2005} have made a similar connection between conservation laws and the analogs of Killing fields for the (hyperbolic) $U/K$-systems, though they do not concern themselves with the universal approach and so do not need the vanishing result of Thm~\ref{thm:MainTheorem}.  There is also a well developed theory of recursion operators  \cite{Olver1977,Sanders2001,Wang2002} that generates conservation laws for many integrable systems.  There is significant overlap between the various approaches.  The recursion $\mcp$ that we introduce in Sec.~\ref{sec:tzitzeica} should be equivalent to the recursion operator given in Sec.~2.16 of \cite{Wang2002}, though the derivation is independent.     

Let us now explain what we mean by the universal perspective.  The characteristic cohomology is a geometric analog of the topological characteristic classes, which we recall now in order to make the analogy plain.  Associated to every smooth manifold $X$ one has its cohomology $H^p(X,\Z)$, which is a measure of the complexity of the topology of $X$.  The extra structure of a complex vector bundle 
\be
\C^r \to E \to X
\ee 
is equivalent to a smooth map $\phi_E:X \to Gr(r,\infty)$ to the infinite Grassmanian \cite{Hatcher}.  The cohomology of $Gr(r,\infty)$ is generated by the Chern classes of the universal bundle 
\be
\C^r \to U \to Gr(r,\infty),
\ee
that is $H^*(Gr(r,\infty),\Z)=\Z[c_1,\ldots,c_r]$ as rings.  Therefore each complex vector bundle $E \to X$ defines a set of {\bf characteristic (cohomology) classes}
\be
\phi_E^*(c_i) \in H^{2i}(X,\Z).
\ee
The $\phi_E^*(c_i)$ vanish if and only if $E \to X$ is a trivial bundle.  Thus the characteristic classes measure the degree of twisting of the vector bundle.  One can also define characteristic classes in terms of connections or other direct approaches  that do not involve the classifying space $Gr(r,\infty)$ \cite{Milnor1974,Bott1982}.  

The characteristic cohomology is an approach to conservation laws of PDE analogous to the universal approach to characteristic classes.  To each exterior differential system $(M,\mci)$ is associated its characteristic cohomology $H^*(M,\Om/\mci)$ \cite{Bryant1995}.  Recall that an integral manifold is an immersed submanifold $\phi:N \to M$ such that $\phi^*(\mci)=0$.  Therefore the characteristic cohomolgy pulls back to any integral manifold to define de Rham cohomology classes of $N$,
\[
\phi^*: \bar{H}^p \to H^p_{dR}.
\]  
These classes should measure how complicated the immersion $\phi:N \to (M,\mci)$ is, \emph{considered as an integral manifold}.  In order to access all of the characteristic cohomology, one must work on the infinite prolongation $(\M{\infty},\I{\infty})$. The analogy is then between $H^*(\M{\infty},\Om^*(\M{\infty})/\mci^{(\infty)})$ and $H^*(Gr(r,\infty),\Z)$.  In the context of variations of Hodge structure this analogy finds a more definite connection \cite{Carlson2009} and raises the question, to what extent can this analogy be made concrete for systems other than those arising in the variation of Hodge structure?

Once one has conservation laws, one would like to use them to better understand the geometric and analytic properties of solutions.  Pinkall and Sterling \cite{Pinkall1989} introduce a notion of finite-type solutions for CMC surfaces in $\R^3$ which was generalized in \cite{Burstall1993}.  We formulate a finite-type condition directly in terms of  conservation laws.  It is equivalent to the Pinkall-Sterling condition for that system. In addition to determining the higher-order characteristic cohomology of the scalar nonlinear Poisson equations, we also present translation invariant representatives of the cohomology classes.  These induce de Rham classes on doubly periodic solutions and should encode significant geometric information about the solution. 

Although we only present results about the Tzitzeica equation in this article ($f_{uu}=\alpha f_u+2\alpha^2 f$), the simpler case $f_{uu}=\beta f$ was carried out by essentially the same method in \cite{Fox2011}, though we hid the role of Killing fields.  These two scalar PDE arise from the Toda field equations for primitive maps to the $6$-symmetric space $\SU(3)/\SO(2)$ and the ($2$-) symmetric space $\SU(2)/\SO(2)$, respectively.  There is one other $k$-symmetric space of rank one, namely, the $4$-symmetric space $\SO(4)/\SO(2)$.  It provides a recursion of order four for the case $f_{uu}=\beta f$ that is essentially the square of the order two recursion that one obtains using the ($2$)-symmetric space $\SU(2)/\SO(2)$.  We have not calculated any examples for $k$-symmetric spaces that don't have $\SO(2)$ as their stabilizer.


\section{The scalar nonlinear Poisson equation as an EDS}
The PDE to be studied is
\begin{equation}\label{eq:fGordon}
\frac{{\partial}^2 u}{\partial z \partial \bar{z}}  = -f(u),
\end{equation}
where $u:\C \to \R$. We encode the PDE as an exterior differential system (EDS) with independence condition. For a basic introduction to EDS see \cite{Bryant1991} or \cite{Ivey2003}. Let $M=J^1(\C,\R)=\C \times \R \times \C$ be the first jet space of maps from $\C$ to $\R$, with coordinates $(z,u,u_0)$ and define the differential forms
\begin{align*}
\zeta&=\ed z\\ 
\omega_1&=\ed u_0 +f \bar{\zeta}\\
\eta_0&=\ed u - u_0 \zeta - \bar{u}_0 \bar{\zeta}\\
\Psi&=\Imag(\zeta \w \omega_1)=-\frac{\mo}{2}(\zeta \w \omega_1-\bar{\zeta} \w \bar{\omega}_1).
\end{align*}
The relevant differential ideal is 
\[\mci=\left\langle \eta_0, \psi  \right\rangle =\langle \eta_0, \zeta \w \omega_1 \rangle.\]
Thus the EDS to be studied is 
\begin{equation}\label{eq:EDS}
(M,\mci),\;\;{\rm where}\;\; M=\C \times \R \times \C \;\; {\rm and}\;\; \mci=\langle \eta_0, \psi \rangle.
\end{equation}
This EDS is involutive with Cartan characters $s_0=1,s_1=2,s_2=0$.  Again, see \cite{Bryant1991,Ivey2003} for the basics of EDS.

As in \cite{Fox2011}, we will work on the infinite prolongation $(\M{\infty},\I{\infty})$.  The structure equations on $\M{\infty}$ are given in terms of the one-forms and functions 
\begin{center}
\begin{tabular}{ll}
 &$\zeta=\ed z$\\
$T^0=f$&$\eta_0=\ed u -u_0\zeta-\bar{u}_0\bar{\zeta}$ \\
$T^{i+1}=\sum_{j=0}^{i}{i \choose j}u_{i-j}T^j_u $&$ \eta_{i+1}=\ed u_{i}+T^{i}\bar{\zeta}-u_{i+1} \zeta$\\
 &$\tau^i=\sum_{j=0}^{i}{i \choose j} T^j_u \eta_{i-j}$,\\
\end{tabular}
\end{center}
for $i \geq 0$.

The real and imaginary parts of $\zeta,\eta_0,\eta_{1},\eta_{2},\ldots$ form a coframe of $\M{\infty}$ and the subbundle
\[
\I{\infty}=\{\eta_0,\eta_{1},\bar{\eta}_{1},\ldots \} \subset \Omega^1(\M{\infty},\C)
\]
generates the (formally Frobenius) differential ideal $\mci^{(\infty)}$. The vector fields on $\M{\infty}$ dual to $\zeta, \eta_0, \eta_i$ are
\begin{align}
&e_{-1}=\dd{\;}{z}+u_0\dd{\;}{u}+\sum_{i=0}^{\infty} u_{i+1}  \dd{\;}{u_{i}}-\sum_{i=0}^{\infty}\overline{T}^{i}\dd{\;}{\bar{u}_{i}}\nonumber \\
&e_0= \dd{\;}{u} \label{eq:vectorfields}\\
&e_i=\dd{\;}{u_{i-1}} \;\;\; i=1 \ldots  \nonumber 
\end{align}

\begin{proposition}
For $i \geq 1$ the following structure equations are satisfied on $\M{\infty}$\upshape{:}
\begin{align*}
\ed T^i &\equiv  T^{i+1}\zeta+ \tau^i  \;\;\; {\rm mod}\; \bar{\zeta}\\
\ed \zeta &=0\\
\ed \eta_0 &=\zeta \w \eta_1 + \bar{\zeta} \w \bar{\eta}_1\\
\ed \eta_{i}&=- \eta_{i+1} \w \zeta+\tau^{i-1} \w \bar{\zeta}
\end{align*}
\end{proposition}

The PDE Equation \eqref{eq:fGordon} is invariant under the $\s{1}$--action $(u,z,\bar{z}) \to (u,\lambda z,\overline{\lambda} \bar{z})$ (with $\lambda \in \C$ and $|\lambda|=1$). This leads to a symmetry of $(\M{k},\mci^{(k)})$, which yields a decomposition of differential forms and thus conservation laws.   To see this, let $F:\s{1} \times \M{k} \to \M{k}$ be defined as 
\begin{equation}\label{eq:Symmetry}
F(\lambda,u,z,u_j)=(u,\lambda^{-1} z,\lambda^{j+1}u_j).
\end{equation}
This induces the following weighting system:
\begin{center}
\begin{tabular}{ll}\label{eq:weights}
$\wt(z)=-1$&$ \wt(\bar{z})=1$\\
$\wt(u_j)=+(j+1)$&$ \wt(\bar{u}_j)=-(j+1)$\\
$\wt(u)=0$ & $\wt(\eta_0)=0$\\
$\wt(\zeta)=-1$ & $\wt(\bar{\zeta})=1$ \\
 $\wt(\eta_j)=+j $&$ \wt(\bar{\eta}_j)=-j$.
\end{tabular}
\end{center}
For $p \geq 0$ and $j \in \Z$ we define the spaces of differential forms of {\bf homogeneous weighted degree} $j$ to be
\begin{equation}\label{eq:wdegforms}
\Omega^p_j(\M{k}) =\left\{\varphi \in \Omega^p(\M{k},\C) \mid F^*\varphi=\lambda^j \varphi \right\}.
\end{equation}
For an element $\varphi \in \Omega^p_j(\M{k},\C)$, we write $\wt(\varphi)=j$. Note that $\wt(\varphi)=-\wt(\overline{\varphi})$. This grading is preserved by exterior differentiation:  
\[
\ed: \Omega^p_j(\M{k}) \to \Omega^{p+1}_j(\M{k}).
\]

The ellipticity of Equation \eqref{eq:fGordon} leads us to the following:
\begin{definition}
Define the subspaces 
\[\Omega^{(1,0)}(\M{\infty})=\C \cdot \{\zeta,\eta_1,\ldots \} \;\text{and}\;\; \Omega^{(0,1)}(\M{\infty})=\C \cdot \{\bar{\zeta},\bar{\eta}_1,\ldots \}
\]
and the operators 
\[
\partial:C^{\infty}(\M{\infty},\C) \to \Omega^{(1,0)}(\M{\infty})\;  \text{and} \;\;\overline{\partial}:C^{\infty}(\M{\infty},\C) \to \Omega^{(0,1)}(\M{\infty})
\]
by the condition
\begin{align*}
\partial A&=e_{-1}(A)\zeta+\sum_{i=1}^{\infty}A_{u_{i-1}} \eta_i \\
\overline{\partial} A&=e_{\overline{-1}}(A)\bar{\zeta}+\sum_{i=1}^{\infty}A_{u_{i-1}} \bar{\eta}_i ,
\end{align*}
It will be convenient to use the linear operator 
\[
J:\Omega^{1}(\M{\infty}) \to \Omega^{1}(\M{\infty}),
\]
which acts by $\mo$ on $\Omega^{(1,0)}(\M{\infty})$, by $-\mo$ on  $\Omega^{(0,1)}(\M{\infty})$, and as the identity on $\R \cdot \eta_0$. This is an almost complex structure on the annihilator of $e_0$. 
\end{definition}

We say that $f$ satisfies an $n^{th}$-order autonomous linear ODE if it satisfies an equation of the form
\begin{equation}\label{eq:linearfirstorderode}
\frac{d^{n}f}{du^{n}}=Z(f,\frac{df}{du},\frac{d^{2}f}{du^{2}},\ldots,\frac{d^{n-1}f}{du^{n-1}})
\end{equation}
where $Z$ is an $\R$-linear function of $n$ variables.  From now on assume that\\
\begin{center}
{{\bf $f$ does not satisfy any first-order linear autonomous ODE.}}
\end{center}
\vspace{.5cm}
By imposing this assumption we rule out the important system $f_u=f$ corresponding to the Liouville equation $u_{z \zb}=e^u$, and the linear systems $f(u)=au+b$ for $a,b \in \R$.  These systems all have infinitely many classical conservation laws and so the machinery we develop is not needed for them.

\section{The characteristic cohomology}
On $(\M{\infty},\mci^{(\infty)})$ the associated characteristic cohomology is 
\[\bar{H}^p:=H^p(\M{\infty},\Omega/\mci),\]
 that is, the cohomology of the quotient complex $(\M{\infty},\Omega/\mci)$.   We will only deal with the local case which, by definition, means that $H^p(\M{k},\R)=0$ for $k \geq 0$ and $p>0$.
\begin{definition}
The space of {\bf higher-order undifferentiated conservation laws} for Eq.~\eqref{eq:EDS} is 
\[
\bar{H}^{1}:=H^{1}(\M{\infty},\Omega/\mci).\]
The space of {\bf higher-order undifferentiated \emph{complex} conservation laws} for $(M,\mci)$ is 
\[\bar{H}^{1}_{\C}:=H^{1}(\M{\infty},\Omega_\C/\mci_\C),\] where the subscript $\C$ denotes complexification. 
\end{definition}

The theory developed by Bryant and Griffiths involves studying conservation laws locally via the isomorphism $H^{1}(\M{k},\Omega/\mci) \cong H^{2}(\M{k},\mcip{k})$, since locally $H^p(\M{k},\R)=0$  for $p>0$.   
\begin{definition}
The space of {\bf higher-order differentiated conservation laws} for $(M,\mci)$ is $H^{2}(\M{\infty},\mci)$. The space of {\bf higher-order differentiated { complex} conservation laws} for $(M,\mci)$ is $H^{2}(\M{\infty},\mci_{\C})$.
\end{definition}
Exterior differentiation provides isomorphisms
\begin{align*}
 \ed:\bar{H}^{1}& \overset{\cong}{\to} H^{2}(\M{\infty},\mci)\\
 \ed:\bar{H}_{\C}^{1}& \overset{\cong}{\to} H^{2}(\M{\infty},\mci_{\C})
\end{align*} 
in the local involutive case.   We have stated this material for the system being studied in this article.  In general, one is not working with forms of degree one.  See \cite{Fox2011} for a summary of the relevant results in \cite{Bryant1991} that are being applied here.

Let $\mcc$ denote the space of differentiated conservation laws in {\it normal form} (see \cite{Fox2011,Bryant1995}).  For the system at hand, $\Phi \in \mcc$ if and only if      
\begin{equation}\label{eq:NormalForm}
\Phi  =\eta_0 \w \rho + A \psi +\sum_{1\leq i < j \leq k } \left( B^{ij}\eta_i \w \eta_j +\overline{B}^{ij}\bar{\eta}_i \w \bar{\eta}_j\right)
\end{equation}
for some $k.$  The one--form $\rho$ and the function $B$ are determined by $A$ via the formulas
\begin{align}
\rho &= -\frac{1}{2} J\ed A  \label{eq:rhobyA}\\
B^{ij}&=\mo \sum_{m=0}^{k-j-i+1}(-1)^{m-i+1}{m+i-1 \choose i-1}(e_{-1})^{m}A_{u_{m+j+i-1}}\label{eq:BbyA}
\end{align}
if we normalize $\rho$ so that $e_0 \lhk \rho=0$.  The function $A$ on $\M{k}$--which we henceforth refer to as the {\bf generating function} of $\Phi$--satisfies 
\begin{equation}\label{eq:NoMixing}
A_{u_i,\overline{u_j}}=A_u=0
\end{equation} 
and   
\begin{equation}\label{eq:ED}
\mce(A):=e_{\overline{-1}} e_{-1}A+f_u A=0.
\end{equation}

\begin{definition}
A conservation law on $\M{\infty}$ in normal form is said to have {\bf level $k$} if it is defined on $\M{k}$.   Let $\mcc_{(k)}$ denote the space of representatives of conservation laws of level $k$ in normal form.  Let $\mcc_d \subset \mcc$ be the subspace of conservation laws in normal form of weighted degree $d$. 
\end{definition}

By studying the normal form of a conservation law we recover the central equation that must be solved in order to obtain conservation laws: 
\begin{equation}
\mce(A):=e_{\overline{-1}} e_{-1}A+f_u A=0.
\end{equation}
See \cite{Vinogradov1989} or \cite{Bryant2003} for the general argument that conservation laws have generating functions that are solutions of the linearization of the PDE being studied.  Pinkall and Sterling \cite{Pinkall1989} refer to these functions as Jacobi fields.

It is a result of Bryant and Griffiths \cite{Bryant1995} that 
\[
\mcc \cong \bar{H}^{n-l}
\]
which in for the system under study means $\mcc \cong \bar{H}^{1}$.

\begin{definition}
Let $V$ be the space of solutions to $\mce(P)=0$ \upshape{(}Eq.~\eqref{eq:ED}\upshape{)} that also satisfy $P_{u_i,\overline{u}_j}=P_u=0$.  Let $V_d \subset V$ be the subspace of homogeneous solutions of weighted degree $d$. 
\end{definition}

The two solutions of Eq.~\eqref{eq:ED} leading to classical conservation laws are $q=zu_0-\zb \ub_0 \in V_0$ and $u_0 \in V_1$.  The following result is proven in \cite{Fox2011}.
\begin{theorem}\label{thm:MainTheorem}
Suppose that $f$ does not satisfy a linear first-order ODE.  Then\upshape{:}
\begin{enumerate}
\item $V_0$ is spanned by $q$. If $d$ is a nonzero even integer, then $V_d=0$. If $d$ is odd, then $\dim_{\C} V_d\leq 1$.
\item For all $d$ we have isomorphisms
\begin{align*}
V_d &\to  \mcc_d.
\end{align*} 
\item $\dim_{\R}(\mcc_{(2n+1)}/\mcc_{(2n)}) =0$.
\item $\dim_{\R}(\mcc_{(2n+2)}/\mcc_{(2n)}) \leq 2$ with equality if and only if $\dim_{\C}(V_{2n+3})=1$.
\end{enumerate}
\end{theorem}

We also have explicit formulas for representatives of a basis for $\bar{H}^1$.  If $P^i$ spans $V_{2i-1}$ for $i \geq 1$, then 
\be\label{eq:clrep}
\p_i= \frac{\mo}{2(2i-1)}J\left( P \ed q- q \ed P  \right)
\ee
represents a nontrivial class, and in fact
\be\label{eq:clbasis}
\bar{H}_{\C}^1=\C \cdot \{[\p_0],[\p_1],[\ol{\p}_1],[\p_2],[\ol{\p}_2],\ldots \}
\ee
where $[\p_i] \in \bar{H}_{\C}^1$ corresponds to $P^i \in V_{2i-1}$ for $i \geq 1$ under the isomorphism $V \cong \bar{H}_{\C}^1$.

For most potentials $f(u)$, no higher-order conservation laws exist (see \cite{Fox2011,Wang2002,Ziber1979}):
\begin{theorem}\label{thm:SecondOrder}Assume that $f$ does not satisfy any linear second-order ODE, i.e.~that $f,f_u$ and $f_{uu}$ are linearly independent over $\R$. Then $V_d=0$ for $|d|\geq 2$, i.e.~no higher-order conservation laws occur.
\end{theorem}

\begin{lemma}\label{lem:PSSolutions}
Suppose that $f_{uu}=\beta f$ with $\beta\neq 0$ and that $f$ does not satisfy any first-order ODE.  Then $\dim_{\C}(V_{2n+1})=1$ for all $n \in \Z$.
\end{lemma}
The proof of this used a recursion derived from formal Killing fields.  A version of this recursion was originally given by Pinkall and Sterling \cite{Pinkall1989}, though they did not link the recursion to Killing fields.  Their astounding recursion formulas do not work for the Tzitzeica equation, which led to the search for other means.  In fact, a recursion operator exists for the Tzitzeica equation, and can be found in the list compiled by Wang \cite{Wang2002}. 

In this article we show how Killing fields can be used to prove the analogue of Lemma \ref{lem:PSSolutions} for the Tzitzeica equation.  It is not hard to show that, for there to exist new conservation laws in normal form at the second prolongation,  then $f$ must satisfy $f_{uu}=\beta f$, and for there to exist new conservation laws in normal form at the fourth prolongation, then $f$ must satisfy either $f_{uu}=\beta f$ or $f_{uu}=\alpha f_u + 2 \alpha^2 f$.  This classification result has appeared in \cite{Ziber1979} and elsewhere.

\begin{example} In the case that $f_{uu}=\alpha f_u + 2\alpha^2 f$ with $\alpha\neq 0$ a coordinate change transforms Equation~\eqref{eq:fGordon} into the Tzitzeica equation $u_{z\bar{z}}=e^{-2u}-e^u$.  The first four spaces $V_{2i+1}$ are
\begin{align*}
& V_1 \cong \C,\quad u_0\in V_1\\
&V_3 \cong 0\\
&V_5 \cong \C,\quad u_{{4}}-5\,\alpha\,u_{{2}}u_{{1}}-5\,{\alpha}^{2}u_{{2}}{u_{{0}}}^{2}-
5\,{\alpha}^{2}{u_{{1}}}^{2}u_{{0}}+{u_{{0}}}^{5}{\alpha}^{4}
 \in V_5\\
&V_7 \cong \C,\quad u_{{6}}-7\,\alpha\,u_{{4}}u_{{1}}-7\,{\alpha}^{2}u_{{4}}{u_{{0}}}^{2}-
14\,\alpha\,u_{{3}}u_{{2}}-28\,{\alpha}^{2}u_{{3}}u_{{1}}u_{{0}}\\
&\quad\qquad\qquad\qquad\qquad-21\,{\alpha}^{2}{u_{{2}}}^{2}u_{{0}}-28\,{\alpha}^{2}u_{{2}}{u_{{1}}}^{2}+
14\,{\alpha}^{3}u_{{2}}u_{{1}}{u_{{0}}}^{2}+14\,{\alpha}^{4}u_{{2}}{u_
{{0}}}^{4}\\
&\quad\qquad\qquad\qquad\qquad+{\frac {28}{3}}\,{\alpha}^{3}{u_{{1}}}^{3}u_{{0}}+28\,{
\alpha}^{4}{u_{{1}}}^{2}{u_{{0}}}^{3}-\frac{4}{3}\,{\alpha}^{6}{u_{{0}}}^{7} \in V_7.
\end{align*}
If we assume that $f(0)=0$, then any nonlinear Poisson equation whose potential satisfies $f_{uu}=\alpha f_u + 2 \alpha^2 f$ can be transformed to $u_{z \zb}=e^{-2u}-e^u$ using a coordinate transformation $u \to a u+b$ and $z \to c z$. 
\end{example}

\begin{rem}
It is curious that there are no conservation laws of weighted-degree three for the Tzitzeica equation, even though there are conservation laws of weighted-degree three for the sinh-Gordon equations, $u_{z \zb}=\sinh(u)$.  It would be interesting to understand the geometric implications of this fact.  
\end{rem}

In this article we complete the story by determining the spaces $V_{2n+1}$ for the Tzitzeica equation. The proof connects Killing fields of primitive maps and conservation laws.  In the next section we outline the basic aspects of primitive maps that will be needed.


\section{Primitive maps}
We recall the notion of primitive maps following \cite{Burstall1995}.  Let $\G$ be the compact real form of the semi-simple complex Lie group $\G^{\C}$ with respect to the anti-holomorphic involution $\sigma:\G^{\C} \to \G^{\C}$.  Let $\tau:\G^{\C} \to \G^{\C}$ be a commuting order-$k$ automorphism, $\tau^k=1$, and $\K \subset \G$ the group it stabilizes.  The tuple $(\G^{\C},\tau,\sigma)$ is a $k$-symmetric space though we often just refer to $M=\G/\K$ as the $k$-symmetric space and it is understood that some triple $(\G^{\C},\tau,\sigma)$ gives rise to $M$.   The automorphisms $\sigma, \tau$ induce automorphisms (known by the same name) of the Lie algebra $\g^{\C}$ of $\G^{\C}$.   We will assume that $\g^{\C} \subset \gl(r,\C)$ so that we can calculate using matrix multiplication.

There is a decomposition of  $\g^{\C}$ into the eigenspaces of $\tau$:
\be
\g^{\C}=\g_{0}\oplus \g_{1}\oplus \ldots \oplus \g_{-1}
\ee 
where the indices are defined modulo $k$ and designate that, for some primitive $k^{th}$-root of unity $\mu$, $\g_{j}$ is the eigenspace with eigenvalue $\mu^j$.  This decomposition induces a decomposition
\be
TM^{\C}= E_{1}\oplus \ldots \oplus E_{-1},
\ee
where each $E_{j}$ is a complex vector bundle, and there is no $E_0$ term because this corresponds to the fiber of $\G \to M$.  
\begin{definition}
Let $N$ be a Riemann surface and $M$ a $k$-symmetric space as above.  A smooth map $\phi:N \to M$ is {\bf primitive} if $\phi_*(T^{(1,0)}N)$ is a complex line in $  E_{-1}$.
\end{definition}
Though $M$ need not be an almost complex manifold, it does have an almost complex structure on a subbundle of $TM$. So a primitive map is something like a pseudo-holomorphic curve.  In \cite{Bryant2006} and \cite{Fox2007} examples of primitive map are referred to as \emph{$CR$-holmorphic curves}. Primitive maps are harmonic for appropriate invariant metrics on $\G /\K$ \cite{Burstall1994}.

On a simply connected Riemann surface $N$, any map to $\phi:N \to M$ can be lifted to a framing $F:N \to \G$ so that $F \cdot \K=\phi$.  The Maurer-Cartan form $\om$ on $\G$ pulls back to be a $\g$-valued one-form $\psi=F^*(\om)$ on $N$.  This form decomposes according to the decomposition of $\g_{\C}$, $\psi=\psi_{0}+\ldots + \psi_{-1}$.  The map is primitive if and only if $\psi=\psi_{-1}+\psi_0+\psi_1$ and $\psi_{-1} \in \Om^{(1,0)}(N)$.  This is easily seen to be equivalent to the condition that there exists a flat family of connections $\psi_{\lambda}=\psi_{-1}\lambda^{-1}+\psi_0+\psi_1 \lambda$ with $\psi_{\lambda} \in \g$ for $\lambda \in \s{1} \in \C$ and $\psi_{-1} \in \Om^{(1,0)}(N)$.  On simply connected surfaces such a family of flat connections is equivalent to a primitive map and choice of framing.


\section{The Toda field equations}\label{sec:Toda}
When $\K$ is a torus, there is a good frame and a good set of coordinates for a primitive map to $\G/ \K$, which reduce the primitive map system to the Toda field equations \cite{Bolton1995,Carberry2011}.   When $\K=\SO(2)$ the system reduces to a scalar nonlinear Poisson equation.  From now on, \\
\begin{center}
{\bf assume that $\K$ is a torus with Lie algebra $\ka$.}
\end{center}
\begin{definition}
Let $U \subset N$ be a simply connected open set of the Riemann surface $N$.  A local frame $F:U \to \G$ is called a {\bf Toda frame} if there exists a complex coordinate $z:U \to \C$ and a smooth map $\Om:U \to \mo \ka$ such that
\be\label{eq:TodaFrame}
F^{-1}\dd{F}{z} = \dd{\Om}{z}+Ad \exp(\Om)(B) \in E_0 \oplus E_{-1}
\ee
where $B \in \sum_{k=0}^r \sqrt{m_k}\xi_{\alpha_k} \in \g_{-1}$, $\alpha_k$ are simple roots, and $\theta=\sum_k m_k \alpha_k$ is the highest root.
\end{definition}
Bolton, Pedit, and Woodward \cite{Bolton1995} prove the following theorems.
\begin{theorem}
Let $\phi:N \to M$ be primitive and let $p_0 \in N$ be such that $\dd{\phi}{z}(p_0)$ is cyclic.  Then there is a local Toda frame $F$ of $\psi$ around $p_0$.  Moreover, the complex coordinate $z$ for the Toda frame is unique up to an arbitrary translation and rotation by a $d_r$-th root of unity, while the Toda frame $F$ is unique up to a multiplication by an element of the center $Z(\G)$ of $\G$. 
\end{theorem}
\begin{theorem}
Let $z:U \to \C$ be a complex coordinate and let $\Om:U \to \mo \mct$ be smooth.  Then Eq.~\eqref{eq:TodaFrame} has a real solution $F$ if and only if $\Om$ satisfies the affine Toda field equations for $\G$, namely
\be\label{eq:TodaEquation}
2 \frac{\del^2\Om}{\del z \del \zb}-\sum_{k=0}^r m_k e^{2\alpha_k(\Om)}\alpha_k^{\#}=0.
\ee
In this case $\phi= \pi \circ F:U \to  M$ is primitive with Toda frame $F$.
\end{theorem}

We will work with the Toda frame for $\SU(3)/\SO(2)$ in Sect.~\ref{sec:tzitzeica}.  Before that we describe the relationship between Killing fields for primitive maps and conservation laws. 


\section{Generating functions from Killing fields}\label{sec:CLfromKF}
We now introduce Killing fields for primitive maps following Burstall and Pedit \cite{Burstall1995}.  Killing fields were introduced in \cite{Burstall1993} as a way of packaging together infinitesimal deformations of harmonic maps.  This fact suggests that, in the case that the primitive maps are related to nonlinear Poisson equations, some components of the Killing field will provide infinitesimal solutions to the linearization of the nonlinear Poisson equation, Eq.~\eqref{eq:ED}.   We will show that, in the case that $\K$ is abelian, the $\g_0$ component of a formal Killing field satisfies the linearization of Toda field equation, making it a generating function for a conservation law.  A similar result should be true if $\K$ is not abelian.  Recall that $\K$ is a torus.

Suppose we have a family of flat connections $\psi_{\lambda}$ on a simply connected Riemann surface $N$ with holomorphic coordinate $z:N \to \C$. The existence of the spectral parameter allows one to use loop groups.  In this article will use the {\bf based loop algebra}
\[
\mcl(\g^{\C})= \left\lbrace  X_{\lambda} \in\g^{\C}[[ \lambda,\lambda^{-1}]] \; : \; {\rm if}\; X_{\lambda}=\sum_{n-\infty}^{\infty}X^n \lambda^n \;{\rm then}\;X^0=0\right\rbrace 
\]
associated to $\g^{\C}$ and the {\bf twisted based loop algebra}
\begin{align*}
\mcl^{\sigma,\tau}(\g^{\C})=\left\lbrace  X_{\lambda} \in \mcl(\g^{\C}):\; \sigma(X_{\ol{\lambda}^{-1}})=X_{\lambda}  \;\; \tau(X_{\mu \lambda})=X_{\lambda}, \;\; \mu=e^{\frac{2\pi \mo}{k}} \right\rbrace
\end{align*}
asociated to the $k$-symmetric space $(\G^{\C},\sigma, \tau)$.  

\begin{definition}
A {\bf formal Killing field} for the family of flat connections $\psi_{\lambda}$ on the Riemann surface $N$ is a map  $X_{\lambda}:N \to \mcl^{\sigma, \tau}(\g^{\C})$ satisfying 
\be\label{eq:KillingEquation}
\ed X_{\lambda}+[\psi_{\lambda}, X_{\lambda}]=0.
\ee
\end{definition}

We can express a Killing field as
\[
X_{\lambda}=\sum_{n=-\infty}^{\infty} \left( X^{kn}\lambda^{kn}+X^{kn+1}\lambda^{kn+1}+\ldots +X^{kn+k-1}\lambda^{kn+k-1}    \right) 
\]
where $X^{kn+j} \in \g_{j}$.  In terms of this expansion Eq.~\eqref{eq:KillingEquation} can be written as
\[
\ed X^{kn+j}+[\psi_0,X^{kn+j}]+[\psi_{-1},X^{kn+j+1}]+[\psi_{1},X^{kn+j-1}]=0,
\]
which decomposes into
\begin{align}
&\del X^{kn+j}+[\psi_0',X^{kn+j}]+[\psi_{-1},X^{kn+j+1}]=0 \label{eq:delX}\\
&\delb X^{kn+j}+[\psi_0'',X^{kn+j}]+[\psi_{1},X^{kn+j-1}]=0.\label{eq:delbX}
\end{align}
where $\psi_0=\psi'_0+\psi''_0$, $\psi'_0 \in \Omega^{(1,0)}(N)$ and $\psi''_0 \in \Omega^{(0,1)}(N)$. 

Choosing a local coframe $\xi \in \Omega^{(1,0)}(U)$  we can express the connection as
\be
\psi=A_{-1} \xi + A_0' \xi - A_{0}'' \xib + A_1 \xib
\ee 
with $A_{j}:N \to \g_j$. We now have the formula
\be
\ed \psi_{-1}=-[A_{-1},A_{0}'']\xi \w \xib,
\ee
which will be used in the calculation below. Define the Laplacian 
\[\Delta:C^{\infty}(N,\g_0) \to C^{\infty}(N,\g_0)
\]
by the condition $\delb \del P= \frac{1}{4}\Delta(P)\xi \w \xib$.   We now show that the $\g_0$-components of the Killing field, $X^{kn}$, satisfy a linear elliptic equation.  

Taking $\delb$ of Eq.~\eqref{eq:delX} for the case $j=0$ and using the fact that $\g_0$ is abelian, we calculate
\begin{align}
 \delb \del X^{kn}& = -\delb[\psi_{-1},X^{kn+1}] \nonumber\\
& =[A_{-1},A_{0}'']X^{kn+1} \xi \w \xib + \psi_{-1} \w (-[\psi_{0}'',X^{kn+1}]-[\psi_{1},X^{kn}]) \nonumber\\
& -([\psi_{0}'',X^{kn+1}]+[\psi_1,X^{kn}]) \w \psi_{-1} - X^{kn+1} [A_{-1},A_{0}''] \xi \w \xib \label{eq:Laplacian}\\
&= \left( [[A_{-1},A_{0}''],X^{kn+1}]+[[A_{0}'',X^{kn+1}],A_{-1}]+[[A_1,X^{kn}],A_{-1}]\right) \xi \w \xib \nonumber\\
& = \left(- [[X^{kn+1},A_{-1}],A_{0}'']+[[A_1,X^{kn}],A_{-1}]\right) \xi \w \xib \nonumber\\
& = [[A_1,X^{kn}],A_{-1}] \xi \w \xib \nonumber
\end{align}
The first term in the penultimate line of Eq.~\eqref{eq:Laplacian} vanishes because $[\g_1 , \g_{-1}] \subset \g_0$ and $\g_0$ is abelian.  This leaves
\be
\frac{1}{4}\Delta X^{kn}+[A_{-1},[A_1,X^{kn}]]=0
\ee

Define the first order linear differential operator $\mcd:\Om^0(N,\g_0) \to \Om^0(N,\g_0)$ as
\be\label{eq:EDforKF}
\mcd(P)= \Delta P+4[A_{-1},[A_{1},P]].
\ee
The calculation above proves
\begin{lemma}\label{lem:g0kf}
The $\g_0$-component of a Killing field lies in the kernel of $\mcd$.   
\end{lemma}
Burstall, et al.~\cite{Burstall1993} prove that every component of a Killing field satisfies such an equation, but we will only need to work with the $\g_0$ component.

We consider the parallel situation for Toda fields.  Using the notation of Sec.~\ref{sec:Toda}, let $\mce_{\Om}:\g_0 \to \g_0$ be the linearization of Eq.~\eqref{eq:TodaEquation} about the solution $\Om$:
\be\label{eq:LinearToda}
\mce_{\Om}(P)=2 \frac{\del^2 P}{\del z \del \zb}+\sum_{k=0}^r m_k e^{\alpha_k(\Om)}\alpha_k^{\#} \alpha_k(P).
\ee 
By differentiating Eq.~\eqref{eq:TodaEquation} with respect to $z$ one finds that $P=\dd{\Om}{z}$ is a solution to Eq.~\eqref{eq:LinearToda}.  The Toda equation, Eq.~\eqref{eq:TodaEquation}, is invariant under the $\s{1}$-action $z \to \lambda z$ for $\lambda \in \s{1} \subset \C$.  As in \cite{Fox2011}, this leads to the solution 
\be
q=2 \Im \left(z \dd{\Om}{z}\right)
\ee 
of Eq.~\eqref{eq:LinearToda}. When $\G/\K=\SU(3)/\SO(2)$, Eq.~\eqref{eq:LinearToda} reduces to Eq.~\eqref{eq:ED} for the Tzitzeica equation.  In Sec. \ref{sec:tzitzeica} we will produce a recursion for the kernal of $\mce_{\Omega}$ using a Toda frame of the $\SU(3)/\SO(2)$ primitive map system.


\section{Killing fields and conservation laws}
Using elements of the kernel of $\mce$, we present a simple formula for conservation laws for primitive maps when $\K$ is abelian.  The formula is essentially the elliptic version of Eq.~5.15 in the work of Terng and Wang on the $G/K$-systems \cite{Terng2005}.  
\begin{definition}
For any pair of functions $P,Q \in \Om^0(N,\g_0)$ define 
\be\label{eq:conservationlaw}
\p_{P,Q}=-\mo J(\kappa(P,\ed Q)-\kappa(Q,\ed P)) \in \Om^1(N,\C)
\ee
where $\kappa$ is the Killing form of $\g^{\C}$.
\end{definition}
\begin{lemma}\label{lem:KFCL}
The one-form $\p_{P,Q}$ is closed on solutions to the primitive map system if $P,Q \in \ker(\mcd)$.
\end{lemma}
\begin{proof}
We calculate, juggling terms with the Killing form at the end to find the desired cancelation.  As before, let $z$ be a local holomorphic coordinate on the simply connected open set $U \subset N$ so that $\ed P=P_z \ed z+P_{\ol{z}} \ed \ol{z}$. We will now express $\psi=A_{-1} \ed z + A_{0}' \ed z + A_{0}'' \ed \zb + A_{1}\ed \zb$.  In the notation from before, this amounts to choosing $\xi=\ed z$.  We will make use of the defining relations $J\ed z=\mo \ed z$ and $J \ed \zb = - \mo \ed \zb$.  First we record that
\begin{align*}
\p_{P,Q}&=-\mo J \left( \kappa(P,\ed Q)-\kappa(Q,\ed P) \right)\\
&=\kappa(P,\del Q)-\kappa(P,\delb Q) - \kappa(Q,\del P)+\kappa(Q,\delb P).
\end{align*}
Now we calculate, leaving the reader the small joy of canceling terms,
\begin{align}
\ed \p_{P,Q}& = \frac{1}{2} \left( \kappa(P,\Delta Q) -\kappa(Q, \Delta P)\right) \ed z \w \ed \zb \nonumber\\
& =-2 \left\{ \kappa(P,[A_{-1},[A_{1},Q]])-\kappa(Q,[A_{-1},[A_{1},P]])   \right\} \ed z \w \ed \ol{z} \nonumber\\
& =-2 \left\{ \kappa([P,A_{-1}],[A_{1},Q])-\kappa(Q,[A_{-1},[A_{1},P]])   \right\} \ed z \w \ed \ol{z}  \\
& =-2 \left\{ \kappa([[P,A_{-1}],A_{1}],Q)-\kappa(Q,[A_{-1},[A_{1},P]])   \right\} \ed z \w \ed \ol{z}  \nonumber\\
& =-2 \left\{ \kappa(Q,[A_{-1},[A_{1},P]])-\kappa(Q,[A_{-1},[A_{1},P]])   \right\} \ed z \w \ed \ol{z} \nonumber \\
& =0.\nonumber
\end{align}
\end{proof}

I expect that most choices of pairs $P,Q$ will lead to trivial conservaton laws- this should be equivalent to the conservation laws commuting under the Poisson bracket.  The results in \cite{Fox2011} suggest that there exists a $q \in \ker(\mce)$, unique up to scale and of degree zero in some sense, such that all of the nontrivial conservation laws are of the form $\p_{P,q}$ where $P \in \ker(\mce)$ and $P$ satisfies some additional conditions.  Up to scale this definition specializes to the normal form found in \cite{Fox2011} using the normal form of Bryant and Griffiths \cite{Bryant1995} together with the $\s{1}$-symmetry to reduce a two-form to a one-form.

\begin{rem}
For rank-$1$ Toda field equations one can define a Poisson bracket on the space of generating functions, or equivalently (see Thm.~\ref{thm:MainTheorem}) on the space of conservation laws \cite{Shadwick1981}.  For $P,Q \in V$ we define $\{ P,Q\}=R$ if $[\p_{P,Q}]=[\p_{q,R}]$. If $P,Q$ are both of nonzero degree, and thus both of odd degree, then $\p_{P,Q}$ is even degree and thus trivial in characteristic cohomology, implying that $\{ P,Q\}=0$.  If $q \in V_0$ and $P \in V_d$ with $d \neq 0$, then clearly $\{q,P \}=P$. Thus we have the following bracket relations:
\begin{align}
\{ P,Q\}&=0 \; {\rm if} \; \deg(P)\neq 0,\;\deg(Q) \neq 0\\
\{ q,P\}&=P \; {\rm if} \; \deg(P) \neq 0
\end{align}
\end{rem}

It seems likely that the conservation laws of the EDS associated to the Toda equations is spanned by $\p_{q,P}$ for a countably infinite set of $P$ and the single exceptional class associated solely to $q=\Im(z\frac{\del \Omega}{\del z})$. This should be nearly the same as the (local) characteristic cohomology for the primitive map system since the two systems are related by integrable extension.

A calculation similar to the proof of Lem.~\ref{lem:KFCL} shows that, given any solution $P$ of \eqref{eq:LinearToda}, $\p_{q,P}$ defines a conservation law for the Toda field system.  We now turn to the Tzitzeica equation to illustrate how Killing fields not only provide generating functions for conservation laws as elements of the characteristic cohomology, but also lead to recursions for proving the existence of the generating functions.  

\section{The Tzitzeica equation and $\SU(3)/\SO(2)$}\label{sec:tzitzeica}
The Tzitzeica equation
\be\label{eq:tzitzeica}
u_{z \zb}=e^{-2u}-e^u,
\ee
is the Toda field equation for primitive maps to $\SU(3)/\SO(2)$.  In the notation of Eq.~\eqref{eq:fGordon} we have $f(u)=e^u-e^{-2u}$ and $f_{uu}=- f_u+2 f$.  The linearization is
\begin{equation}
\mce(P)=P_{-1,\ol{-1}}+\left(e^u+2e^{-2u} \right) P.
\end{equation}
In this section we prove
\begin{theorem}\label{thm:Main}
For the nonlinear Poisson equation, Eq.~\eqref{eq:tzitzeica} and $k \geq 0$,
\begin{align*}
V_{0} &\cong \C\\
V_{6k+1} &\cong \C\\
V_{6k+3}&=0\\
V_{6k+5} &\cong \C \\
V_{-d}& \cong \ol{V_d} \;\;\;{\rm for}\;\; d  \geq 0.
\end{align*}  
\end{theorem}
Although a recursion operator is known for the Tzitzeica equation \cite{Wang2002}, the analogue of the vanishing result $V_{6k+3}=0$ seems not to be in the literature.  We prove Thm.~\ref{thm:Main} by finding a pair of recursions, 
\begin{align}
\mcp:&V_{d}\to V_{d+6}\\
\mcn:&V_{d+6} \to V_{d}
\end{align}
for the generating functions of the conservation laws of the Tzitzeica equation. 
\begin{definition}{\bf (The Recursion $\mcp$)}:
Let $a^n \in V_{d}$ for $d \equiv 1  \mod 2$.  Define $a^{n+1}=\mcp(a^n) \in  V_{d+6}$ through the following process.  Define the one-form
\be\label{eq:alpha}
\alpha^n= \frac{\mo}{\sqrt{2}}(a^n_{-1,-1}+2u_0 a^n_{-1}) \ed z -\frac{3\mo}{\sqrt{2}}e^u a^n \ed \zb
\ee
Let $b^n:\M{\infty} \to \C$ be the unique function of weighted-degree $d+1$ such that $\ed b^n \equiv \alpha^n$ modulo $\I{\infty}$.  Now recursively construct
\begin{align}
f^n&=\mo e^{\frac{u}{2}}(b^n_{-1} - u_{0} b^n)\label{eq:fn}\\
r^n&= \sqrt{2}e^{-\frac{u}{2}}(f^n_{-1}+\frac{1}{2} u_{0} f^n)\\
s^n&=-\frac{1}{\sqrt{2}} e^{-{u}} r^n_{-1}.\label{eq:sn}
\end{align} 
Define the one-form
\be\label{eq:beta}
\beta^n= \mo e^u(s^n_{-1,-1}-u_0 s^n_{-1}) \ed z  - 3 \mo e^{-u}s^n \ed \zb
\ee
and let $t^n:\M{\infty} \to \C$ be the unique function of weighted-degree $d+5$ such that $\ed t^n \equiv \beta^n$ modulo $\I{\infty}$.   Finally, setting 
\be
a^{n+1}=- \sqrt{-2}(t^n_{-1}+u_0 t^n)
\ee
we define $\mcp(a^n)=a^{n+1}$.
\end{definition}

\begin{definition}{\bf (The Recursion $\mcn$)}:
Let $a^{n+1}\in V_{d}$.  We define $a^{n}=\mcn(a^{n+1}) \in  V_{d-6}$ through the following process.  First define the one-form
\be
\alpha^n= \frac{\mo}{\sqrt{2}}\left[ 3e^u a^n \ed z -(a^n_{\ol{-1},\ol{-1}}+2\ol{u_0} a^n_{\ol{-1}}) \ed \zb \right]
\ee
Let $t^n:\M{\infty} \to \C$ be the unique function of weighted-degree $d+1$ such that 
\[
\ed (e^u t^n) \equiv \alpha^n \;\; {\rm mod}\; \I{\infty}.
\]
Now recursively construct
\begin{align}
s^n&=\mo e^{u} t^n_{\ol{-1}} \\
r^n&= \sqrt{2}(s^n_{\ol{-1}}+\ol{u_{0}} s^n)\\
f^n&=-\frac{1}{\sqrt{2}} e^{-\frac{u}{2}} r^n_{\ol{-1}}.
\end{align} 
Define the one-form
\begin{equation}
\beta^n=  -3 \mo \left(e^{-\frac{3}{2}u}f^n \right) \ed z+  \mo e^u \left[  \left(e^{-\frac{1}{2}u} f^n \right)_{\ol{-1},\ol{-1}}-\ol{u_0} \left(e^{-\frac{1}{2}u }f^n \right)_{\ol{-1}} \right]  \ed \zb
\end{equation}
and, once again, let $B^n:\M{\infty} \to \C$ be the unique function of weighted-degree $d-5$ such that $\ed \left( e^{-u}b^n \right) \equiv \beta^n$ modulo $\I{\infty}$.   Finally, setting
\be
a^{n}=\sqrt{2}\mo e^{-u} b^n_{\ol{-1}}
\ee
we define $\mcn(a^{n+1})=a^{n}$.
\end{definition}

The maps $\mcp$ and $\mcn$ have been normalized so that 
\begin{align}
\mcp(u_d+ \cdots)&=u_{d+6}+ \cdots \nonumber\\
\mcp(u_{\ol{d}}+ \cdots)&=u_{\ol{d-6}}+ \cdots \label{eq:normalization}\\
\mcn(u_{{d}}+ \cdots)&=u_{{d-6}}+ \cdots  \nonumber\\
\mcn(u_{\ol{d}}+ \cdots)&=u_{\ol{d+6}}+ \cdots \nonumber
\end{align}
and thus $\mcp \mcn= \mcn \mcp = {\rm Id}$ on $V_{d}$ for $d \neq -5,-1,0,1,5$. 

The importance of these recursions lies in the following:
\begin{proposition}\label{prop:recursion2}
\hspace{1cm}
\begin{enumerate}
\item The map $\mcp:V_{d} \to V_{d+6}$ is a linear isomorphism (as vector spaces) for all odd $d$ except when $d=-1$ or $d=-5$.
\item The map $\mcn:V_{d} \to V_{d-6}$  is a linear isomorphism (as vector spaces) for all odd $d$ except when $d=1$ or $d=5$.
\end{enumerate}
\end{proposition}
Together these recursions provide the
\begin{proof}[Proof of Thm. \ref{thm:Main}]
Starting with $a^0=u_0$, $\mcp^j(a^0)$ will generate bases for $V_{1+6k}$ for $k \geq 0$. Starting with 
\[
a^0 = u_{{4}}+5u_{{2}}u_{{1}}-5u_{{2}}{u_{{0}}}^{2}-
5{u_{{1}}}^{2}u_{{0}}+{u_{{0}}}^{5},
\]
$\mcp^j(a^0)$ will generate bases for $V_{5+6k}$ for $k \geq 0$.

One can check directly that $V_3=0$. Using induction this implies that $V_{6k+3}=0$: For if $P \in V_{6k+3}$ were a nonzero element, then $\mcn(P) \in V_{6k-3}$ would be nonzero.  By applying $\mcn$ repeatedly we would produce a nonzero element of $V_3$, a contradiction.  The $V_{d}$ for $d<0$ are determined by the isomorphism $V_{-d} \cong \ol{V_d}$. 
\end{proof}

The proofs that $\mcp$ and $\mcn$ are well defined and the proof of Thm.~\ref{prop:recursion2} will be presented through a series of Lemmas (see Lemmas~\ref{lem:btdefined} to \ref{lem:nou}).   Before giving the proofs we will derive $\mcp$ and $\mcn$ using the Killing field equation for primitive maps to $\SU(3)/\SO(2)$.  

To begin, we introduce the $6$-symmetric space structure on $\SU(3)/\SO(2)$.  Let $R \in \SO(3)$ be the rotation by $\frac{2 \pi}{6}$ about the axis through $(1,0,0) \in  \C^3$.  On $\SL(3,\C)$ define
$$
\tau(g)=R \trp(g)^{-1} R^{-1}.
$$
This is an order $6$-automorphism.    One computes that on $\sla(3,\C)$ it has eigen spaces
\begin{align}
\g_0 &=\left\{ \bp 0&0&0 \\ 0&0& -A \\ 0& A &0\ep : A \in \C  \right\}  \nonumber \\
\g_1 &=\left\{ \bp 0&- B &-\mo B \\  B & C &-\mo C \\ \mo B &-\mo C &- C \ep : B, C  \in \C  \right\}  \nonumber  \\
\g_2 &=\left\{ \bp 0& F&-\mo F \\  F&0& 0 \\ -\mo F&0 &0\ep  : F \in \C  \right\} \label{eq:eigenspaces} \\
\g_3 &=\left\{ \bp -2R & 0&0 \\ 0&R& 0 \\ 0 &0 &R \ep  : R \in \C  \right\}  \nonumber \\
\g_4 &=\left\{ \bp 0& S& \mo S \\ S&0& 0 \\ \mo S&0 &0 \ep : S \in \C  \right\}   \nonumber \\
\g_5 &=\left\{ \bp 0&- T& \mo T \\ T& V & \mo V \\ -\mo T& \mo V &- V \ep : T,V  \in \C   \right\} . \nonumber  
\end{align}
Thus the stabilizer is the $\SO(2) \subset \SU(3) $ that fixes $(1,0,0)$.

We define the $\su(3,\C)$-valued connection on $\M{\infty}$ 
\begin{align}
\psi&=\psi_{-1}+\psi_{1}+\psi_0\nonumber\\
&=\frac{1}{2}\bp 0 &-\zeta_1 &\mo \zeta_1 \\ \zeta_1 &\mo \zeta_2 &-\zeta_2 \\ -\mo \zeta_1& -\zeta_2&-\mo \zeta_2 \ep+\frac{1}{2}\bp 0 &-\zetab_1 &-\mo \zetab_1 \\ \zetab_1 &\mo \zetab_2 &\zetab_2 \\ \mo \zetab_1& \zetab_2 & -\mo \zetab_2  \ep\\
&+\bp 0 &0&0 \\ 0& 0 & -\mo \theta \\ 0& \mo \theta & 0 \ep\nonumber
\end{align} 
where $\zeta_1=\sqrt{2}e^{\frac{u}{2}} \ed z$, $ \zeta_2 = e^{-u}\ed z$, and $\theta=\frac{1}{2}(u_{\ol{0}} \ed \zb - u_0 \ed z)$.  One checks that $\psi_{\lambda}=\lambda^{-1}\psi_{-1}+\lambda\psi_{1}+\psi_0$ is flat modulo the ideal $\I{\infty}$ for any $\lambda \in \C^*$.

Let $X_{\lambda}:\M{\infty} \to \sla(3,\C)$ be a formal Killing field for the $\SU(3)/\SO(2)$ primitive map system, that is, it satisfies Eq.~\eqref{eq:KillingEquation}.  We are working with a $6$-symmetric space so we have 
\be
X_{\lambda}=\sum_{n=-\infty}^{\infty}\left( X^{6n}\lambda^{6n}+X^{6n+1}\lambda^{6n+1}+\cdots+X^{6n+4}\lambda^{6n+4}+X^{6n+5}\lambda^{6n+5}\right)
\ee
with $X^{kn+j}$ taking values in the $\g_j$ specified in Eq.~\eqref{eq:eigenspaces}. We parametrize a Killing field, according to the eigenspaces in Eq.~\eqref{eq:eigenspaces}, as 
\begin{align*}
X^{6n} &=\bp 0&0&0 \\ 0&0& -a^n \\ 0& a^n &0\ep  \nonumber \\
X^{6n+1} &=\bp 0&- e^{-\frac{u}{2}}b^n &-\mo e^{-\frac{u}{2}}b^n \\  e^{-\frac{u}{2}}b^n & e^{u}c^n &-\mo e^{u}c^n \\ \mo e^{-\frac{u}{2}}b^n &-\mo e^{u}c^n &- e^{u}c^n \ep   \nonumber  \\
X^{6n+2} &= \bp 0& f&-\mo f \\  f&0& 0 \\ -\mo f&0 &0\ep  \\
X^{6n+3} &= \bp -2r & 0&0 \\ 0&r& 0 \\ 0 &0 &r \ep    \nonumber \\
X^{6n+4} &= \bp 0& e^{\frac{u}{2}}s^n& \mo e^{\frac{u}{2}}s^n \\ e^{\frac{u}{2}}s^n&0& 0 \\ \mo e^{\frac{u}{2}}s^n&0 &0 \ep \nonumber \\
X^{6n+5} &= \bp 0&- e^{\frac{u}{2}}t^n& \mo e^{\frac{u}{2}}t^n \\ e^{\frac{u}{2}}t^n& v^n & \mo v^n \\ -\mo e^{\frac{u}{2}}t^n& \mo v^n &- v^n \ep  . \nonumber  
\end{align*}
The factors of $e^{ku}$ with $k \in \Z[\frac{1}{2}]$ are inserted for convenience.  One computes that Eq.~\eqref{eq:KillingEquation} decomposes into the $\ed z$
\begin{align}
a^n_{-1}&+\mo \sqrt{2}b^n-2c^n=0 \label{eq:abc}\\
b^n_{-1}&-u_0 b^n+\mo e^{-\frac{u}{2}}f^n=0 \label{eq:bf}\\
c^n_{-1}&+2u_0 c^n +\sqrt{2}e^{-\frac{u}{2}}f^n=0 \label{eq:cf}\\
f^n_{-1}&+\frac{1}{2}u_0 f^n-\frac{3}{\sqrt{2}} e^{\frac{u}{2}}r_n=0 \label{eq:f}\\
r^n_{-1}&+\sqrt{2}e^u s^n=0  \label{eq:r}\\
s^n_{-1}&-\sqrt{2}v^n+\mo e^{-u}t^n=0 \label{eq:stv} \\
t^n_{-1}&+u_0 t^n -\frac{\mo}{\sqrt{2}}a^{n+1}=0 \label{eq:ta}\\
v^n_{-1}&-u_0 v^n-e^{-u}a^{n+1} =0\label{eq:va}
\end{align} 
and $\ed \zb$ components
\begin{align}
a^n_{\ol{-1}}&-\mo\sqrt{2}e^u t^{n-1}+2e^{-u}v^{n-1}=0  \label{eq:abcbar}\\
b^n_{\ol{-1}}&+\frac{\mo}{\sqrt{2}} e^{u}a^n=0 \label{eq:babar}\\
c^n_{\ol{-1}}&+ e^{-2u}a^n=0  \label{eq:cfbar}\\
f^n_{\ol{-1}}&-\frac{1}{2}\ol{u_0}f^n+ \mo e^{\frac{-3u}{2}}b^n-\sqrt{2}e^{\frac{3u}{2}}c^n=0 \label{eq:fbar} \\
r^n_{\ol{-1}}&+\sqrt{2}e^{\frac{u}{2}}f^n=0 \label{eq:rbar}\\
s^n_{\ol{-1}}&+\ol{u_0}s^n-\frac{3}{\sqrt{2}}r^n=0  \label{eq:stvbar}\\
t^n_{\ol{-1}}&+\mo e^{-u}s^n =0 \label{eq:tabar}\\
v^n_{\ol{-1}}&+\ol{u_0}v^n+\sqrt{2}e^u s^n =0\label{eq:vabar}
\end{align} 
Recall that $a^n_{-1}=e_{-1}(a^n)$ and $a^n_{\ol{-1}}={\ol{e_{-1}}}(a^n)$ where $e_{-1}$ is the vector field defined in Eq.~\eqref{eq:vectorfields}.  

We begin with the recursion $\mcp$.  First, Eq.~\eqref{eq:abc} can be arranged as
\be\label{eq:c}
c^n= \frac{1}{2}(a^n_{-1}+\mo \sqrt{2} b^n).
\ee
Eliminating $f^n$ from Eq. \eqref{eq:bf} and Eq. \eqref{eq:cf} results in
\be
{\mo}\sqrt{2}(b^n_{-1} - u_0 b^n)+( c^n_{-1} + 2 u_0 c^n)=0.
\ee
Substituting into this the expression for $c^n$ from Eq.~\eqref{eq:c}, we obtain 
\be
b^n_{-1}=\frac{\mo}{3\sqrt{2}}(a^n_{-1,-1}+2u_0 a^n_{-1}).
\ee
From the $\delb$ equations we use Eq.~\eqref{eq:babar}
\be
b^n_{\ol{-1}}=-\frac{\mo}{\sqrt{2}}e^u a^n.\nonumber
\ee
The two expressions for the derivatives of $b^n$ are, up to a factor of $3$, equivalent to
\be\label{eq:tzitzeicabfroma}
\ed b^n \equiv \alpha^n \;\; {\rm mod}\; \I{\infty} \nonumber
\ee
for
\be
\alpha^n = \frac{\mo}{\sqrt{2}}(a^n_{-1,-1}+2u_0 a^n_{-1}) \ed z -\frac{3\mo}{\sqrt{2}}e^u a^n \ed \zb.
\ee
Eqs.~\eqref{eq:bf},\eqref{eq:f},\eqref{eq:r} are, up to a factor of $3$ in the second of them, 
\begin{align}
f^n&=\mo e^{\frac{u}{2}}(b^n_{-1} - u_0 b^n)\\
r^n&= \sqrt{2}e^{-\frac{u}{2}}(f^n_{-1}+\frac{1}{2} u_0 f^n)\\
s^n&=-\frac{1}{\sqrt{2}} e^{-{u}} r^n_{-1}.
\end{align}
Eq.~\eqref{eq:stv} can be expressed as
\begin{equation}\label{eq:v}
v^n=\frac{1}{\sqrt{2}} (\mo e^{-u}t^n+s^n_{-1})
\end{equation}
Eliminating $a^{n+1}$ from Eq.~\eqref{eq:ta} and  Eq.~\eqref{eq:va} leads to 
\begin{equation}
\mo \sqrt{2} \left(t^n_{-1}  + u_0 t^n\right)+e^u\left( v^n_{-1}-u_0 v^n \right)=0.
\end{equation}
Substituting the expression for $v^n$ from Eq.~\eqref{eq:v} into this collapses the equation down to
\be
t^n_{-1}=\frac{\mo}{3}e^u(s^n_{-1,-1}-u_0 s^n_{-1}).
\ee
This expression for $t^n_{-1}$ along with Eq.~\eqref{eq:tabar} suggests that we define
\be
\beta^n= \mo e^u(s^n_{-1,-1}-u_0 s^n_{-1}) \ed z  - 3\mo e^{-u}s^n \ed \zb.
\ee
Then the expressions for the derivatives of $t^n$ are equivalent to 
\be
\ed t^n \equiv \beta^n  \;\; {\rm mod}\; \I{\infty}.
\ee

Finally, Eq.~\eqref{eq:ta}, can be expressed as
\be
a^{n+1}=- \sqrt{-2}(t^n_{-1}+u_0 t^n).
\ee

To summarize, the Killing field equations can be arranged as a recursion for the sequence
\begin{equation}
\cdots \to a^n \to b^n \to f^n \to r^n \to s^n \to t^n \to a^{n+1} \to \cdots
\end{equation}
Producing $b^n$ from $a^n$ and $t^{n}$ from $s^n$ hinge on $[\alpha^n]$ and $[\beta^n]$ vanishing in characteristic cohomology.   The other stages only require differentiation.  The definitions imply that at each stage the weighted-degree increases by one (we set all additive constants to be zero so that the functions are weighted-homogeneous).  Thus $\wt(a^{n+1})=\wt(a^n)+6$.   

We now present lemmas which will prove that the map $\mcp$ is well defined and an isomorphism for the appropriate degrees.  We will use the notation $A \in  \langle u_j \rangle$ to mean that the function $A$ on $\M{\infty}$ only involves the variables $u_0,u_1,\ldots$ but not $u,\ub_0,\ub_1,\ldots$ Similar notation will be used to indicate any other dependence or independence that is relevant.   
\begin{lemma}\label{lem:btdefined}
If $\mce(a^n)=0$ then $b^n$ and $t^n$ are well-defined weighted homogeneous functions on $\M{\infty}$. 
\end{lemma}

Thus $\mcp: \Omega^0(\M{\infty},\C) \to \Omega^0(\M{\infty},\C)$ is well-defined.
\begin{lemma}\label{lem:ainkernel}
If $a^n \in \ker(\mcd)$ and $a^{n+1}=\mcp(a^n)$, then $\mcd(a^{n+1})=0$.
\end{lemma}

Together Lem.~\ref{lem:btdefined} and \ref{lem:ainkernel} imply that
\[
\mcp:\ker(\mcd) \to \ker(\mcd)
\]
is well-defined.  What is left is to show that if $a^n \in V_d$ then $\mcp(a^n) \in V_{d+6}$.  Heading in this direction, we have
\begin{lemma}\label{lem:btexist}
We have the following results on $b^n$ and $t^n$ \upshape{:}
\begin{enumerate}
\item If $0 \neq a^n \in V_d$ for $d>0$  then $0 \neq b^n \in \langle u_j \rangle$ is a weighted homogeneous polynomial of degree $d+1$.
\item If $0 \neq a^n \in V_d$ for $d<-5$ then $0 \neq e^{-u}b^n \in \langle \ub_j \rangle$ is a weighted homogeneous polynomial of degree $d+1$. 
\item If $0 \neq a^n \in V_d$ for $d>0$ then $0 \neq t^n \in \langle u_j \rangle$ is a weighted homogeneous polynomial of degree $d+5$.
\item If $0 \neq a^n \in V_d$ for $d<-5$ then $0 \neq e^{u}t^n \in \langle \ub_j \rangle$ is a weighted homogeneous polynomial of degree $d+5$. 
\end{enumerate}
\end{lemma}
Lem.~\ref{lem:btexist} in conjunction with the recursion formula now imply that 
\[\mcp(a^n)=\sum_{l=1}^n \left(e^{k_l u}A^l \right).
\]
However we also know that $\mcd(a^{n+1})=0$.
\begin{lemma}\label{lem:nou}
Suppose that
\[a^n=\sum_{l=1}^n \left(e^{k_l u}A^l \right)
\]
with each $A^l \in  \langle u_j,\ub_j \rangle$ a weighted homogeneous polynomial, $k_l \in \Z$, and $k_m \neq k_l$ if $k \neq l$.  If $\mcd(a^n)=0$ then $a^n_u=0$ and $a^n$ is either a polynomial exclusively in $u_j$ or exclusively in $\ub_j$.  
\end{lemma}
Thus together, Lem.~\ref{lem:btdefined} to~\ref{lem:nou} prove Prop.~\ref{prop:recursion2}.  We now give the proofs of Lem.~\ref{lem:btdefined} to~\ref{lem:nou}

\begin{proof}[Proof of Lem.~\ref{lem:btdefined}]
The Killing field equations Eq.~\eqref{eq:abc} to \eqref{eq:vabar} imply that $\ed \alpha^n \equiv \ed \beta^n \equiv 0 \; \mod \; \I{\infty}$.   Thus $[\alpha^n], [\beta^n] \in \bar{H}^1$.  However, because $\wt(a^n)=d$ is odd, $\wt(\alpha^n)=d+1$ is even.  Similarly $\wt(\beta^n)=d+5$ is also even.  The vanishing result Thm.~\ref{thm:MainTheorem} implies that there are no conservation laws of even weighted degree.  Thus there exist unique weighted-homogeneous $b^n,t^n:\M{\infty} \to \C$ with $\wt(b^n)=d+1$ and $\wt(t^n)=d+5$ such that $\ed b^n \equiv \alpha^n$ mod $\I{\infty}$ and $\ed t^n \equiv \beta^n$ mod $\I{\infty}$.
\end{proof}

\begin{proof}[Proof of Lem.~\ref{lem:ainkernel}]
Suppose $a^n \in \ker(\mce)$ and that $b^n,f^n,r^n,s^n,t^n,a^{n+1}=\mcp(a^n)$ are defined by the recursion.  A direct computation of $a^{n+1}_{-1,\ol{-1}}$ involves a rather large number of terms.  To cut down on the algebraic complexity, we will instead derive equations satisfied by $s^n$.  Using these one may compute that $a^{n+1} \in\ker(\mce)$.  The first goal is to show that 
\begin{equation}
s^n_{\ol{-1}}=-u_{\ol{0}}s^n+3\mo b^n_{-1,-1}-3\mo \left(u_1+u_0^2 \right)b^n.\label{eq:sn-1}
\end{equation} 
This can be obtained as follows.  The condition $\ed b^n \equiv \alpha^n$ allows us to express $b^n_{\ol{-1}}$ and $a^n_{-1,-1}$ in terms of $a^n,a^n_{-1},b^n_{-1}$.  Then by repeatedly applying $\ed^2=0$ one obtains the expressions:
\begin{align*}
b^n_{-1,\ol{-1}}&=-\frac{3}{\sqrt{2}}\mo e^u \left( u_0 a^n +a^n_{-1}\right)\\
b^n_{-1,{-1},\ol{-1}}&=-\frac{3}{\sqrt{2}}\mo e^u \left( u_1+u_0^2 \right)a^n -3e^u b^n_{-1}\\
b^n_{-1,{-1},{-1},\ol{-1}}&= (u_0^2+u_1)b^n_{-1}+\left(  u_2+2u_0 u_1   \right)b^n +\mo e^u s^n.
\end{align*}
Then Eqs.~\eqref{eq:fn}-~\eqref{eq:sn} can be solved for $b^n_{-1,-1,-1}$ to obtain
\begin{align*}
b^n_{-1,{-1},{-1}}&= (u_0^2+u_1)b^n_{-1}+\left(  u_2+2u_0 u_1   \right)b^n +\mo e^u s^n.
\end{align*}
From this and the expressions for the derivatives of $b^n$ one can derive Eq.~\eqref{eq:sn-1}.  Now applying $\ed^2=0$ consecutively to $s^n,s^n_{-1},s^n_{-1,-1}$ one obtains the relations:
\begin{align*}
s^n_{-1,\ol{-1}}&=-u_{\ol{0}}s^n_{-1}-(2e^u+e^{-2u})s^n\\
s^n_{-1,{-1},\ol{-1}}&=-\ub_0 s^n_{-1,-1}-(e^u+2e^{-2u})s^n_{-1}+2u_0(-e^u+e^{-2u})s^n\\
s^n_{-1,{-1},{-1},\ol{-1}}&=-\ub_0 s^n_{-1,-1,-1}-3e^{-2u}s^n_{-1,-1}+u_0(6e^{-2u}-3e^u)s^n_{-1}\\
&+\left( 2u_1 (-e^u+e^{-2u})-2u_o^2(e^u+2e^{-2u})\right)s^n.
\end{align*}
The recursion provides expressions for $\ed t^n$ in terms of $s^n$ and its derivatives.  Using the expressions just obtained for the derivatives of $s^n$ along with the definition of $a^{n+1}$ in terms of $t^n$ allows one to check that $a^{n+1} \in\ker(\mce)$.   We leave the details to the reader. 
\end{proof}

\begin{proof}[Proof of Lem.~\ref{lem:btexist}]
First assume that $a^n \in V_d$ with $d>0$ so that $b^n \in V_{d+1}$.  We will prove that $b^n \in \langle u_j \rangle$.  By Lemma 8.13 in \cite{Fox2011}, if $a^n \in V_{d}$ with $d>0$ then $a^n=u_{d-1}+\cdots \in \langle u_j \rangle$.  First we show that $b^n_z=b^n_{\bar{z}}=0$.  Differentiating Eq.~\eqref{eq:alpha} with respect to $z$ results in 
\[
\ed \left( \dd{b^n}{z} \right) \in \I{\infty},
\]
using $\dd{\eta_i}{z}=\dd{\eta_i}{\bar{z}}=0$.  But there are no exact forms in the ideal and so $\dd{{b}^n}{z}$ is constant.  Together with a similar argument involving $\bar{z}$ this implies that ${b}^n$ is at most linear in $z$ or $\bar{z}$ so we can write ${b}^n=b_0^n+k_1 z + k_2 \bar{z}$ where $b_0^n$ is independent of $z$ and $\bar{z}$.  However, because the right hand side of Equation~\eqref{eq:beta} does not contain any terms of the form $k_1 \ed z$ or $k_2 \ed \bar{z}$, we must have $k_1=k_2=0$.  Thus ${b}^n$ is independent of $z$ and $\bar{z}$.  

Now write $b^n=\sum_{l=0}^{\infty}b^{n,l}{\ub}_{2n}^l$ where $b^{n,l}$ is independent of ${\ub}_{2n}$ as well as $\ub_{2n}$, the variable of least (`most negative') weighted degree on which $b^n$ depends.   The $\del$ part of $\ed b^n \equiv \alpha^n$ can be written as
\[
\sum_{l=0}^{\infty}\left(  b^{n,l}l {\ub}_{2n}^{l-1} {\ub}_{2n+1}+b^{n,l}_{{-1}} {\ub}_{2n}^l \right)=\frac{\mo}{\sqrt{2}}\left( a^n_{-1,-1} + 2 u_0 a^n_{-1}\right). 
\]
The lowest variable appearing in $b^n$ is ${\ub}_{2n}$ and it does not appear in $a^n,a^n_{-1},a^n_{-1,-1}$ because  $a^n, a^n_{-1},a^n_{-1,-1} \in  \langle u_j \rangle$.  By induction we find that $b^n$ is independent of ${\ub}_{2n}$.  Then using induction again we find that $b^n$ is independent of ${\ub}_j$ for all $j \geq 0$.   Thus $b^n \in  \langle u,u_j \rangle$.  Expanding $b^n$ as a power series in $u$ and considering the $\delb$-components of $\ed b^n \equiv \alpha^n$ shows that $b^n \in  \langle u_i \rangle$.  (The appearance of $e^{u}$ on the right hand side does not spoil the argument because it is the image of the $\bar{u}_0 \dd{\;}{u}$ term in $e_{\overline{-1}}$ that leads to the vanishing of $b^n_u$.)   This completes the proof of the first part.  

The proof of the second part is similar by slightly more involved.  So suppose that $a^n \in V_{d}$ for $d < -6$.  By Lemma 8.13 in \cite{Fox2011}, $a^n=\ub_{d-1}+\cdots \in \langle \ub_j \rangle$.  We now use this to argue that $b^n=e^{u}\bh^n$ where $\bh^n \in \langle \ub_j \rangle$ is a weighted homogenous polynomial of degree $\wt(a^n)+1$. 

The same argument as used above shows that  $b^n_z=b^n_{\bar{z}}=0$.   Now write $b^n=\sum_{l=0}^{\infty}b^{n,l}{u}_{2n}^l$ where $b^{n,l}$ is independent of ${u}_{2n}$ and $u_{2n}$ is the variable of highest weighted degree on which $b^n$ depends.  We see that $a^n \in  \langle \ub_j \rangle$ implies that $a^n_{-1} \in  \langle u,\ub_j \rangle$ and $a^n_{-1,-1} \in  \langle \ub_j,u,u_0 \rangle$.  The $\del$ part of $\ed b^n \equiv \alpha^n$ can be written as
\[
\sum_{l=0}^{\infty}\left(  b^{n,l}l {u}_{2n}^{l-1} {u}_{2n+1}+b^{n,l}_{{-1}} {u}_{2n}^l \right)=\frac{\mo}{\sqrt{2}}\left( a^n_{-1,-1} + 2 u_0 a^n_{-1}\right). 
\]
The highest variable appearing in $b^n$ is ${u}_{2n}$ and it does not appear in $a^n,a^n_{-1},a^n_{-1,-1}$.  By induction we find that $b^n$ is independent of ${u}_{2n}$.  Then using induction again we find that $b^n$ is independent of ${u}_j$ for all $j \geq 0$.   Thus $b^n \in  \langle u,\ub_j \rangle$.  

We can now determine the explicit dependence of $b^n$ on $u$.  The $\delb$ component of Eq.~\eqref{eq:beta} is 
\be
b^n_{\ol{-1}}= \frac{3\mo}{\sqrt{2}} e^{u}a^n.
\ee
Using the identity  $[\frac{\del}{\del u}, e_{\ol{-1}}]b^n=0$, which follows from the fact that $b^n \in \langle u,\ub_j \rangle$, we compute
\begin{align*}
(b^n_u)_{\ol{-1}}=\frac{\del}{\del u} b^n_{\ol{-1}}=\frac{\del}{\del u} (  \frac{3\mo}{\sqrt{2}} e^{u} a^n)=\frac{3\mo}{\sqrt{2}} e^{u} a^n= b^n_{\ol{-1}}.
\end{align*}
Thus we have $e_{\ol{-1}}(b^n_u-b^n)=0$, which implies that, up to an additive constant, $b^n=e^{u}\bh^n$ with $\bh^n \in \langle \ub_j \rangle$, as desired. Requiring that $b^n$ be weighted homogeneous fixes the constant to be zero.

The proof of the statements regarding $t^n$ follow parallel reasoning and so are omitted.
\end{proof}

\begin{proof}[Proof of Lem.~\ref{lem:nou}]
Suppose that $u_p$ is the variable of highest weighted degree in $a^n$.  Suppose that $\ub_q$ is the variable of lowest weighted degree that occurs with the term containing the highest power of $u_p$.  Then $e_{\ol{-1}}e_{-1}$ on this term will produce a term with factor $u_{p+1}\ub_{q+1}$ multiplying the maximal possible powers of $u_p$ and $\ub_q$.  In the equation $\mcd(a^n)=0$, no other term will cancel this by the assumption that we chose a term with the maximal exponents.  By induction we find that $a^n_{u_i \ub_j}=0$.  It is impossible for $a^n$ to be homogeneous, to involve both $u_j$ and $\ub_j$, and to satisfy $a^n_{u_i \ub_j}=0$.  Thus $a^n \in  \langle u_j \rangle$ if $\wt(a^n)>0$ and $a^n \in  \langle \ol{u_j} \rangle$ if $\wt(a^n)<0$.

We will now argue that if $a^{n} \in \langle u, \ub_j \rangle $ and $\mcd(a^n)=0$ then in fact $a^n \in  \langle \ub_j \rangle$. The proof for the positive degree situation is analogous and so it is omitted.   We calculate that 
\begin{align*}
a^{n+1}_{-1,\ol{-1}}&=u_0 \sum_{i=1}^n k_le^{k_l u } \left[k_l \ol{u_0} A^l + A^l_{\ol{-1}} \right]+\cdots
\end{align*}
where we have shown only the terms involving $u_0$.  Thus the condition $\mcd(a^n)=0$ implies that 
\be\label{eq:}
k_le^{k_l u } \left(k_l \ol{u_0} A^l + A^l_{\ol{-1}} \right)=0
\ee 
for each $l=1, \ldots, n$.  However, as $A^l \in  \langle \ol{u_j} \rangle$ is a polynomial, it follows that $A^l=0$ if $k_l \neq 0$.  Thus $a^n=A^0$ is independent of $u$ and $a^n \in \langle \ol{u_j} \rangle$.
\end{proof}

The proof for $\mcn$ follows a parallel line of reasoning and so it is omitted.  We mention that whereas it is easier to prove properties of $\mcp$ for $V_d$ with $d>0$, it is easier to prove properties of $\mcn$ for $V_d$ with $d<0$.

It should be possible to generalize the arguments in this section to obtain a complete description of the characteristic cohomology of all Toda field equations.


\section{Finite-type solutions}
Pinkall and Sterling used the notion of finite-type solutions to reduce the PDE $u_{z \zb}=-\frac{1}{2}\sinh(2u)$ to a system of ODE on a finite dimensional manifold.  This was re-interpretted by Burstall et al.\ in \cite{Burstall1993}, using Killing fields, for all harmonic maps into symmetric spaces.  We give a definition of finite-type solutions in terms of the characteristic cohomology that recaptures the notion in Pinkall and Sterling's work and should also be equivalent to that of Burstall et al.\ \cite{Burstall1993}.    

Here is a general definition of finite-type solutions in terms of conservation laws.  
\begin{definition}\label{defn:generalfinitetype}
Let $(\M{\infty},\I{\infty})$ be the infinite prolongation of an involutive EDS with an infinite set of conservation laws $\{ [\p_1], [\p_2], [\p_3], \ldots \}$ and normal forms $\p_i \in \Omega^1(\M{\infty},\R)$.  Then an integral manifold $\iota:N \to M$ is of {\bf finite-type} if for all $m>n$ there exist $a_i \in \R$ such that 
\begin{equation}
\iota^*(\p_{m}-\sum_{i=1}^n a_i \p_i)=0.
\end{equation}  
\end{definition}

This definition applies to the nonlinear Poisson equations (Toda field equations) studied here and in \cite{Fox2011}.  For a given nonlinear Poisson equation, let $P^i=u_{2i-2}+ \cdots$ be a basis for $V_{2i-1}$ when $V_{2i-1} \neq 0$, for $i=1,2,\ldots$  Let $\pt_i=qP^i_{-1}\zeta+q_{\ol{-1}}P^i\zetab$ for $i \in \N$.  The relationship 
\[
\pt_i=(2i-1)\p_i+\frac{1}{2}\ed (Pq)
\]
and Eq.~\eqref{eq:clbasis} (see \cite{Fox2011} for the proof) imply that the $[\pt_i]$ form a basis for the conservation laws of Eq.~\eqref{eq:EDS}.  Using the isomorphisms $\bar{H}^1\cong \mcc \cong V$, $\pt_i$ corresponds to $P^i \in V_{2i-1}$ .  Let $\iota^{(0)}:N \to M$ be an integral manifold and $\iota:N \to \M{\infty}$ the prolonged map. Then Def.~\ref{defn:generalfinitetype} applies, though we restate it using the notation suited to this example.
\begin{definition}\label{defn:finitetype}
The integral submanifold $\iota^{(0)}:N \to M$ is of {\bf finite-type} $g \in \N$ if $g$ is the lowest natural number such that for all $g' >g$, $\iota^*(\pt_{g'})=\sum_{j=1}^{g}(a_j \iota^*(\pt_j)+b_j \iota^*(\ol{\pt}_j))$ for some complex numbers $a_j,b_j$. 
\end{definition}
The isomorphism $\ol{\mcc_{d}}=\mcc_{-d}$ corresponds to $\ol{V_{d}}=V_{-d}$, so that conservation laws of both positive and negative weighted-degree are coming into play.  

This definition is not quite satisfying as part of the theory of the characteristic cohomology.  The definition relies on choosing representatives of (characteristic) cohomology classes.  While this seems unavoidable, there is not yet a satisfying notion of normal form for elements of $\bar{H}^1$, but only for elements of $H^2(M,\mci)$.  While the normal form $\p_{P} \equiv \frac{\mo}{2d}J(P \ed q - q \ed P)$ (see \cite{Fox2011} Sec.~7) that is derived from the normal form $\Phi_P \in \mcc \cong H^2(M,\mci)$ is  a candidate, it is desirable to have a definition of the normal form that is independent of $\Phi_P \in \mcc$.  The normal form $\p_P$ is also undesirable in that it will not descend to a torus domain for doubly periodic solutions.  In Sec.~\ref{sec:translationinvariant} we introduce a translation invariant representative $\ph_P$, though the definition is rather patchwork.     

Pinkall and Sterling made the following definition, which we state using the notation of this article, including the notation presented immediately prior to Defn.~\ref{defn:finitetype}.
\begin{definition}\label{defn:finitetypePS}
The integral submanifold $\iota^{(0)}:N \to M$ is of {\bf finite-type} $g \in \N$ if $g$ is the lowest natural number such that for all $g' >g$, $\iota^*(P^{g'})=\sum_{j=1}^{g}(a_j \iota^*(P^j)+b_j \iota^*(\ol{P}^j))$ for some complex numbers $a_j,b_j$. 
\end{definition}
It is easy to see that  
\begin{lemma}
Definitions~\ref{defn:finitetype} and~\ref{defn:finitetypePS} agree.
\end{lemma}
\begin{proof}
The form $\pt^i$ is linear in $P^i$.  Thus any linear relations on the $P^i$ extend to linear relations on $\pt^i$. Conversely, by examining the $\zetab$ coefficient,  linear relations on the $\pt^i$ imply linear relations on the $P^i$.  
\end{proof}

As noted in \cite{Pinkall1989}, doubly periodic solutions, i.e. solutions $u:\C \to \R$ that descend to a torus $T^2$, are always of finite-type.   This follows from the following two facts: 1)Linear elliptic operators on compact manifolds have finite dimensional kernels and 2)The $P^i$ all pull back to be in the kernel of the linear elliptic operator, $\frac{\del^2}{\del z \del \zb}+f_u$ (See Eq.~\eqref{eq:ED}).  

\begin{example}
The simplest system with which to illustrate the implications of finite-type conditions is that of holomorphic curves in $\C^2$. For this system we need not resort to prolongation.  For a given holomorphic curve we may choose holomorphic coordinates $z,w$ on $\C^2$ so that almost everywhere the holomorphic curve can be graphed as $(z,s(z))$, where $s(z)$ is an analytic function.  The finite-type condition will imply that $s(z)$ satisfies a rational ODE.  We now explain this.

The EDS for holomorphic curves is $(M,\mci)$ with $M=\C^2$ and $\mci=\langle \Omega^{(2,0)}\oplus \Omega^{(0,2)} \rangle$.  The space of conservation laws is just the space of holomorphic $(1,0)$-forms plus the space of anti-holomorphic $(0,1)$-forms.   The forms $\{z^i w^j \ed z , z^iw^j \ed w\}$ and their complex conjugates are a basis for the space of polynomial conservation laws.  Define an integral manifold of finite-type to be a holomorphic curve $\gamma:N \to \C^2$ for which there exists a conservation law $\p=f(z,w)\ed z + g(z,w)\ed w$, with $f$ and $g$ polynomial, such that
\be\label{eq:holomorphicfinitetype}
\gamma^*(\p)=0.
\ee  
Near a smooth point we have that either $\gamma^*(\ed z) \neq 0$ or $\gamma^*(\ed w) \neq 0$.  Suppose that $\gamma^*(\ed z) \neq 0$ so that locally $w=s(z)$.  Then Eq.~\eqref{eq:holomorphicfinitetype} becomes
\be
(f+g\frac{ds}{dz})\ed z=0.
\ee
The PDE has been reduced to the ODE 
\be
\frac{ds}{dz}=-\frac{f(z,s(z))}{g(z,s(z))}
\ee
where $f$ and $g$ are polynomial.  Thus $s$ satisfies a rational ODE.  
\end{example}

A more powerful version of this is known for primitive maps to $k$-symmetric spaces.  If $\gamma:\R^2 \to \G/\K$ is a primitive map of finite-type then it can be solved by a pair of commuting ODE on a finite dimensional manifold \cite{Burstall1993}.  The result above, that holomorphic curves of finite-type are solutions of rational ODE, is a simple example of the elaborate theory used to reduce primitive maps of finite-type to an ODE.     


\section{Translation invariant conservation laws}\label{sec:translationinvariant}
Unfortunately, the presence of $z$ in $q$ prevents $\p_{q,P}$ from being translation invariant, impeding it from defining a cohomology class on a doubly periodic solution corresponding to CMC tori in a $3$-dimensional space form or special Legendrian tori in $\s{5}$.   We now remedy this situation by finding a translation invariant gauge.   Our starting point is Eq.~\eqref{eq:conservationlaw} for $\K \cong \SO(2)$,\footnote{The definition of $\p_P$ in (Eq. (19)) of \cite{Fox2011} is equivalent to $\frac{1}{d}\p_{q,P}$ in the notation Eq.~\eqref{eq:conservationlaw} if $\wt(P)=d$. }
\begin{equation}
\p_{P,Q}:=-\mo J (P\ed Q -Q \ed P).
\end{equation}
A short computation shows that 
\[
\p_{P,Q} \equiv -2 \left( Q \del P + P \delb Q\right) + \ed (QP) \; \mod \I{\infty}.
\]
Thus $[\p_{P,Q}]=[-2 \left( Q \del P + P \delb Q\right)]$ as elements of $\bar{H}^1$.  It will be convenient to compute with $\pt_{P,Q}:=Q \del P + P \delb Q$ which, up to scale, has the same class as $[\p_{P,Q}]$.  We calculate that
\be
\pt_{P,q}=z \pt_{P,u_0}- \zb\pt_{P,\ub_0}- \ub_0 P \zetab.
\ee
We would like to find $G: \M{\infty} \to \C$ such that 
\be\label{eq:phihat}
\ph=\pt_{P,q}-\ed G
\ee
is translation invariant, that is
$$\mcl_{\dd{\;}{z}}\ph=\mcl_{\dd{\;}{\zb}}\ph=0.$$

The weighted degrees of $\pt_{P,u_0}$ and $\pt_{P,\ub_0}$ are even, so they are trivial characteristic cohomology classes.   Therefore there exist functions $A,B \in C^{\infty}(\M{\infty},\C)$ such that 
\begin{align}
\pt_{P,u_0} & = \ed A +  \alpha  \label{eq:dA}\\
\pt_{P,\ub_0} & =  \ed B+ \beta   \nonumber
\end{align}
with $\alpha,\beta \in  \I{n}$.  

We calculate that
\be
\mcl_{\dd{\;}{z}} \p_{P,u_0} =0
\ee
implying that
\be
\ed (\dd{A}{z})+\mcl_{\dd{\;}{z}} \alpha=0.
\ee
We find that $\mcl_{\dd{\;}{z}} \alpha \in \I{n}$ since $\mcl_{\dd{\;}{z}} \eta_i=0$ for all $i$, and so conclude that 
\be
\ed (\dd{A}{z}) \in \I{n}.
\ee
However, $\I{n}$ has trivial infinite derived system, so it must be that $\ed (\dd{A}{z})=0$.  Similarly we conclude that $\ed (\dd{A}{\zb})=0$ and thus 
$$A=\hat A+a_1z +a_2 \zb$$
where $\hat A$ is a function independent of  $z,\zb$ and $a_1,a_2 \in \C$.  The same argument works to derive the corresponding result for $B$.   Eqn.~\eqref{eq:dA} now implies that $a_1=a_2=0$ because neither $\p_{u_0,P}$ nor $\alpha$ have any terms of the form $a_1 \ed z+a_2 \ed z$ for $a_1,a_2 \in \C$.  Thus $A$ is independent of $z,\zb$.  The same argument applies to $B$. 

Let $G=z A - \zb B$ in Eq.~\eqref{eq:phihat}.  Then
\be
\ph \equiv -A \ed z+(B - \ub_0 P) \zetab \;\; {\rm mod} \; \I{\infty}.
\ee
As an example, if one takes $P=u_0$ in the case that $f_{uu}=\alpha f_u + 2 \alpha^2 f$ (the Tzitzeica equation) then 
\[
\ph \equiv -\frac{1}{2}u_0^2 \zeta -\left( 2|u_0|^2+\beta^{-1}(f_u-\alpha f) \right) \ed \zetab 
\]
is a translation invariant representative.

This discussion proves
\begin{lemma}  
Let $(M,\mci)$ be a nonlinear Poisson system with $\dim(\bar{H}^1) = \infty$ (but $f$ does not satisfy any first-order ODE).  Let $\iota^{(0)}:N \to M$ be an integral manifold and $\iota:N \to \M{\infty}$ the prolonged map.  Suppose that $\iota$ is doubly periodic with respect to the lattice $\Lambda \subset \C$.  Then $\iota$ induces a map on cohomology
\be
\iota^*:\bar{H}^1 \to H^1_{dR}(\C/\Lambda,\C).
\ee
\end{lemma}
The meaning of these cohomology classes seems to be completely unexplored.
\bibliographystyle{amsplain}
\bibliography{bibliography}

\def\cprime{$'$}
\providecommand{\bysame}{\leavevmode\hbox to3em{\hrulefill}\thinspace}
\providecommand{\MR}{\relax\ifhmode\unskip\space\fi MR }
\providecommand{\MRhref}[2]{%
  \href{http://www.ams.org/mathscinet-getitem?mr=#1}{#2}
}
\providecommand{\href}[2]{#2}
\begin{thebibliography}{10}

\bibitem{Hatcher}
{Allen Hatcher}, \emph{Vector {B}undles and {K}-{T}heory},
  http://www.math.cornell.edu/~hatcher/VBKT /VBpage.html.

\bibitem{Bolton1995}
John Bolton, Franz Pedit, and Lyndon Woodward, \emph{Minimal surfaces and the
  affine {T}oda field model}, J. Reine Angew. Math. \textbf{459} (1995),
  119--150. \MR{1319519 (96f:58040)}

\bibitem{Bott1982}
Raoul Bott and Loring~W. Tu, \emph{Differential forms in algebraic topology},
  Graduate Texts in Mathematics, vol.~82, Springer-Verlag, New York, 1982.
  \MR{658304 (83i:57016)}

\bibitem{Bryant2006}
Robert~L. Bryant, \emph{Second order families of special {L}agrangian 3-folds},
  Perspectives in {R}iemannian geometry, CRM Proc. Lecture Notes, vol.~40,
  Amer. Math. Soc., Providence, RI, 2006, pp.~63--98. \MR{2237106
  (2007e:53063)}

\bibitem{Bryant1991}
Robert~L. Bryant, Shiing~Shen Chern, Robert~B. Gardner, Hubert~L. Goldschmidt,
  and Phillip~A. Griffiths, \emph{Exterior differential systems}, Mathematical
  Sciences Research Institute Publications, vol.~18, Springer-Verlag, New York,
  1991. \MR{1083148 (92h:58007)}

\bibitem{Bryant1995}
Robert~L. Bryant and Phillip~A. Griffiths, \emph{Characteristic cohomology of
  differential systems. {I}. {G}eneral theory}, J. Amer. Math. Soc. \textbf{8}
  (1995), no.~3, 507--596. \MR{1311820 (96c:58183)}

\bibitem{Bryant2003}
Robert~L. Bryant, Phillip~A. Griffiths, and Daniel Grossman, \emph{Exterior
  differential systems and {E}uler-{L}agrange partial differential equations},
  Chicago Lectures in Mathematics, University of Chicago Press, Chicago, IL,
  2003. \MR{1985469 (2004g:58001)}

\bibitem{Burstall1994}
F.~E. Burstall and F.~Pedit, \emph{Harmonic maps via {A}dler-{K}ostant-{S}ymes
  theory}, Harmonic maps and integrable systems, Aspects Math., E23, Vieweg,
  Braunschweig, 1994, pp.~221--272. \MR{MR1264189}

\bibitem{Burstall1995}
\bysame, \emph{Dressing orbits of harmonic maps}, Duke Math. J. \textbf{80}
  (1995), no.~2, 353--382.

\bibitem{Burstall1993}
Francis~E. Burstall, Dirk Ferus, Franz Pedit, and Ulrich Pinkall,
  \emph{Harmonic tori in symmetric spaces and commuting {H}amiltonian systems
  on loop algebras}, Ann. of Math. (2) \textbf{138} (1993), no.~1, 173--212.
  \MR{1230929 (94m:58057)}

\bibitem{Carberry2011}
Emma Carberry and Katherine Turner, \emph{Harmonic tori in de sitter spaces
  $s^{2n}_1$}, arXiv:1201.5696 (2011).

\bibitem{Carlson2009}
James Carlson, Mark Green, and Phillip Griffiths, \emph{Variations of {H}odge
  structure considered as an exterior differential system: old and new
  results}, SIGMA Symmetry Integrability Geom. Methods Appl. \textbf{5} (2009),
  Paper 087, 40. \MR{2559674 (2011b:14028)}

\bibitem{Fox2011}
{Daniel Fox and Oliver Goertsches}, \emph{Higher-order conservation laws for
  the nonlinear poisson equation via characteristic cohomology}, Selecta Math.
  (N.S.) (2011).

\bibitem{Fox2007}
Daniel Fox, \emph{Coassociative cones ruled by 2-planes}, Asian J. Math.
  \textbf{11} (2007), no.~4, 535--553. \MR{2402937 (2009f:53072)}

\bibitem{Ivey2003}
Thomas~A. Ivey and Joseph~M. Landsberg, \emph{Cartan for beginners:
  differential geometry via moving frames and exterior differential systems},
  Graduate Studies in Mathematics, vol.~61, American Mathematical Society,
  Providence, RI, 2003. \MR{2003610 (2004g:53002)}

\bibitem{Vinogradov1989}
I.~S. Krasil{\cprime}shchik and A.~M. Vinogradov, \emph{Nonlocal trends in the
  geometry of differential equations: symmetries, conservation laws, and
  {B}\"acklund transformations}, Acta Appl. Math. \textbf{15} (1989), no.~1-2,
  161--209, Symmetries of partial differential equations, Part I. \MR{1007347
  (91b:58277)}

\bibitem{Milnor1974}
John~W. Milnor and James~D. Stasheff, \emph{Characteristic classes}, Princeton
  University Press, Princeton, N. J., 1974, Annals of Mathematics Studies, No.
  76. \MR{0440554 (55 \#13428)}

\bibitem{Olver1977}
Peter~J. Olver, \emph{Evolution equations possessing infinitely many
  symmetries}, J. Mathematical Phys. \textbf{18} (1977), no.~6, 1212--1215.
  \MR{0521611 (58 \#25341)}

\bibitem{Pinkall1989}
Ulrich Pinkall and Ivan Sterling, \emph{On the classification of constant mean
  curvature tori}, Ann. of Math. (2) \textbf{130} (1989), no.~2, 407--451.
  \MR{1014929 (91b:53009)}

\bibitem{Sanders2001}
Jan~A. Sanders and Jing~Ping Wang, \emph{Integrable systems and their recursion
  operators}, Proceedings of the {T}hird {W}orld {C}ongress of {N}onlinear
  {A}nalysts, {P}art 8 ({C}atania, 2000), vol.~47, 2001, pp.~5213--5240.
  \MR{1974732 (2004e:37109)}

\bibitem{Shadwick1981}
W.~F. Shadwick, \emph{The {H}amilton {C}artan formalism for {$r$}th-order
  {L}agrangians and the integrability of the {K}d{V} and modified {K}d{V}
  equations}, Lett. Math. Phys. \textbf{5} (1981), no.~2, 137--141. \MR{612421
  (82k:58051)}

\bibitem{Terng2005}
Chuu-Lian Terng and Erxiao Wang, \emph{Curved flats, exterior differential
  systems, and conservation laws}, Complex, contact and symmetric manifolds,
  Progr. Math., vol. 234, Birkh\"auser Boston, Boston, MA, 2005, pp.~235--254.
  \MR{MR2105152 (2005k:58005)}

\bibitem{Ziber1979}
A.~V. vZiber and Alexei~B. vSabat, \emph{The {K}lein-{G}ordon equation with
  nontrivial group}, Dokl. Akad. Nauk SSSR \textbf{247} (1979), no.~5,
  1103--1107. \MR{550472 (80k:35060)}

\bibitem{Wang2002}
Jing~Ping Wang, \emph{A list of {$1+1$} dimensional integrable equations and
  their properties}, J. Nonlinear Math. Phys. \textbf{9} (2002), no.~suppl. 1,
  213--233, Recent advances in integrable systems (Kowloon, 2000). \MR{1900197
  (2003e:37099)}

\end{thebibliography}
\end{document}